\documentclass[12pt]{article}
\usepackage[utf8]{inputenc}
\usepackage[english]{babel}
\usepackage[pdftex]{color}
\usepackage[pdftex]{graphicx}
\usepackage{url}
\usepackage{amsmath,amssymb,amsthm}
\usepackage{float}

\usepackage[a4paper,left=15mm,right=15mm,top=30mm,bottom=20mm]{geometry}
\newtheorem{theor_eng}{Theorem}[section]
\newtheorem{proposition}[theor_eng]{Proposition}
\newtheorem{lemma}[theor_eng]{Lemma}
\newtheorem{corollary}[theor_eng]{Corollary}
\newtheorem{remark}[theor_eng]{Remark}
\newtheorem{example}[theor_eng]{Example}

\begin{document}

\title{Affine weighted Motzkin paths and the differential kernel method}

\author{
Alexander Omelchenko\\
\small Constructor University Bremen, Campus Ring 1, 28759 Bremen, Germany\\
\small\tt aomelchenko@constructor.university
}

\date{}
\maketitle

\begin{abstract}
We study Motzkin paths whose up-, level-, and down-step weights are affine
functions of the height, \(\alpha_k=Ak+\alpha_0\), \(\gamma_k=Bk+\gamma_0\),
\(\beta_k=Ck+\beta_0\).  Let \(w_{n,k}\) be the total weight of paths of
length \(n\) ending at height \(k\), and let \(W_n=w_{n,0}\) be the return
column.  For affine weights the row recurrence becomes a first-order
differential equation in the variable marking terminal height, and, unless
\(\beta_0=C\), that equation involves the unknown return series as boundary
data.  We show that the return series is nevertheless determined by the step
rule alone: a characteristic ending on the floor forces a cancellation and
yields an Abel--Volterra equation for the returns, while the same identity
read at an interior point reconstructs the entire triangle \((w_{n,k})\).
We call this boundary cancellation the \emph{differential kernel method}; it
takes the place of substituting an admissible kernel root, which is
unavailable because the equation is differential rather than algebraic in
the catalytic variable.  When the boundary index \(\nu=(\beta_0-C)/C\) is a
positive integer \(m\), the construction becomes finite and combinatorial:
shifting every height by \(m\) turns the divided-difference evolution into
an ordinary weighted path model on a half-line carrying \(m\) virtual levels
below the visible floor, and a first-entry decomposition expresses
\((w_{n,k})\) through two such local models.  Exactly one bridge, of weight
\(\alpha_0-A\), leads from the virtual strip back to the visible region; it
is closed, and the terminal-height columns factorise, precisely when
\(\alpha_0=A\).  For the Dyck weights \(\alpha_k=k+1\),
\(\beta_k=k+\nu+1\), \(\gamma_k=0\) this gives
\(\sum_{n\ge0}w_{n,k}t^n/n!=\sec^{\nu+1}t\,\tan^kt\).


\end{abstract}


\section*{Introduction}

Weighted lattice paths are a basic tool of enumerative combinatorics, and
Motzkin paths provide one of the simplest settings in which a boundary, a
terminal-height refinement, and nonconstant step weights interact
\cite{Stanley2,Flajolet2009,Krattenthaler2015}.  We study Motzkin paths whose
up-, level-, and down-step weights are affine functions of the height.  If
\(w_{n,k}\) denotes the total weight of paths of length \(n\) ending at
height \(k\), then the two principal objects are the full triangle
\((w_{n,k})\) and its return column \(W_n=w_{n,0}\).  The purpose of the
paper is to determine the return column from the affine step rule itself and
to recover the complete terminal-height refinement from it.  When the weights
are positive integers they may be read as multiplicities, or colours, of the
corresponding arrows; Figure~\ref{fig:intro-motzkin} shows the ordinary
Motzkin triangle beside a triangle with height-dependent multiplicities.

\begin{figure}[H]
\centering
\includegraphics[width=.70\textwidth]{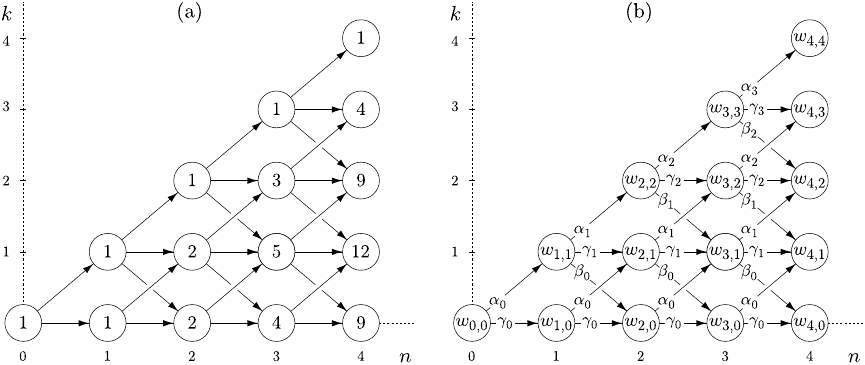}
\caption{The classical Motzkin triangle and a Motzkin triangle with
height-dependent multiplicities.}
\label{fig:intro-motzkin}
\end{figure}

Throughout, \emph{return series} is the combinatorial name for
\(W(t)=\sum_{n\ge0}W_nt^n/n!\); the analytic term \emph{boundary trace} is
used only inside proofs, where \(W\) appears as the unknown boundary value of
\(w(x,t)=\sum_{n,k\ge0}w_{n,k}x^kt^n/n!\) at \(x=0\).

Several classical theories describe parts of this problem.  First-return
decomposition gives Jacobi continued fractions and the moment interpretation
of weighted Motzkin excursions
\cite{Flajolet1980CF,Viennot1985,Chihara1978,Ismail2005}.  The kernel method
and equations with a catalytic variable determine unknown boundary series in
many lattice-path and map-enumeration problems
\cite{BanderierFlajolet2002,Prodinger2004,BousquetMelouJehanne2006}.
Exponential Riordan arrays describe triangles whose height columns
factorise \cite{Barry2011,BarryMwafise2018}.  Affine Dyck and Motzkin
triangles with multiplicities were studied in \cite{Meshkov2010}, which
derived exponential generating functions for several linear weight rules and
already isolated the difficulty caused by an unknown axis series.  That paper
also contains the concrete rule
\[
   \alpha_k=k+1,\qquad \beta_k=k+2,\qquad \gamma_k=k+1,
\]
solved there by adjoining one auxiliary lower row and subtracting its
contribution at the end; in the language used below this is the one-level
no-return construction, and Section~\ref{sec:integer-shift} extends it to an
arbitrary integral height shift while distinguishing exactly between the case
in which a single final deletion is valid and the case in which paths can
leave the visible levels and return.
These viewpoints motivate the present work, but none of them by itself
selects the return series for the general affine rule and simultaneously
reconstructs every terminal height.

For affine weights, the row recurrence becomes a first-order differential
equation in the variable marking terminal height.  In the balanced case
\(\beta_0=C\), this equation is local; that theory is developed separately
in \cite{OmelchenkoBalanced}.  When \(\beta_0\ne C\), a divided-difference
term contains the unknown return series \(W(t)\).  Solving the differential
equation along characteristics therefore leaves a genuine boundary problem:
one must determine \(W(t)\) from the path model, rather than prescribe it as
additional data, and then recover the full bivariate generating function.

Our first result solves this problem for positive boundary index
\(\nu=(\beta_0-C)/C\).  Along a characteristic ending at the floor, an
explicit homogeneous factor vanishes.  Finiteness of the path-generating
function forces the remaining term to cancel and yields a one-dimensional
Abel--Volterra equation for \(W(t)\).  Reading the same characteristic
identity at an interior point gives a reconstruction formula for the entire
height-refined triangle.  We call this boundary-cancellation step the
\emph{differential kernel method}: the equation to be solved is differential
rather than algebraic in the catalytic variable, so no kernel root is
available, and the boundary limit along a characteristic plays the role that
the substitution of an admissible root plays in the classical kernel method.
The characteristic calculation itself is standard. Characteristics and Volterra equations
also occur in linear birth--death models \cite{ZhengChaoJi2004}; the new
point here is the path-refined formulation, in which the same identity both
selects the return column and reconstructs all terminal heights.

The integer values of the boundary index have an additional combinatorial
meaning.  If \(\nu=m>0\), shifting the height scale by \(m\) turns the
divided-difference evolution into a local auxiliary evolution on a half-line
with \(m\) virtual levels.  The shifted copy of the original triangle and
the freely evolved auxiliary triangle are distinct objects.  A first-entry
decomposition relates them through two local shifted path triangles, gives
the return equation at the floor, and after \(m\) differentiations produces
a Volterra equation of the second kind.  The only bridge from the virtual
strip back to the visible region has weight \(\alpha_0-A\).  Its vanishing
is exactly the condition under which the virtual terms may be deleted only
once, after the final row has been formed.

Terminal-height factorisation is even more rigid.  If arbitrary
height-dependent weights satisfy \(H_k=WY^k\) for every height and
\(H_1\not\equiv0\), then the column equations force
\[
   \alpha_k=(k+1)\alpha_0,\qquad
   \gamma_k=\gamma_0+k(\gamma_1-\gamma_0),\qquad
   \beta_k=\beta_0+k(\beta_1-\beta_0).
\]
Thus factorisation itself characterises the arrival-proportional affine
rule.  Inside the affine family this reduces to \(\alpha_0=A\), the same
condition that closes the bridge in the integer shifted model.  The two
criteria therefore arise independently and coincide.

The last part of the paper places these results in two supporting
frameworks.  Edge rescaling shows that the return sequence remembers the
level weights and the products of opposite edge weights, but forgets their
orientation; this leads to the standard Jacobi continued fraction and the
affine Hankel product.  A linear birth--death model with a single defect at
the floor exhibits the same boundary mechanism and identifies the boundary
index with a ratio of transition rates, explaining why the continuous
parameter is fundamental and the finite shifted-floor picture is
exceptional.  Classical Riordan, continued-fraction, and special-function
formulas are used as exact descriptions of these structures, rather than as
the source of the main results.

Section~\ref{sec:affine-motzkin} defines the weighted paths and isolates the
boundary problem.  Section~\ref{sec:continuous-kernel} derives the return
equation, reconstructs the full triangle, and works out an early Dyck
example.  Section~\ref{sec:integer-shift} develops the integer shifted model,
its first-entry decomposition, and the bridge criterion.
Section~\ref{sec:collapse} proves the exact terminal-height factorisation
criterion.  Section~\ref{sec:external-readings} gives the Jacobi and
boundary-perturbed birth--death interpretations, and
Section~\ref{sec:conclusion} summarises the structural conclusions.


\section{Affine weighted Motzkin paths and the boundary problem}
\label{sec:affine-motzkin}

A Motzkin path starts at height \(0\), never goes below height \(0\), and
uses three types of steps:
\[
   k\to k+1,\qquad k\to k,\qquad k+1\to k.
\]
We assign height-dependent weights as follows.  An up-step leaving height
\(k\) has weight \(\alpha_k\), a level-step at height \(k\) has weight
\(\gamma_k\), and a down-step arriving at height \(k\), equivalently a step
from height \(k+1\) to height \(k\), has weight \(\beta_k\).  The weight
of a path is the product of the weights of its steps; the empty path has
weight \(1\).

Let \(w_{n,k}\) be the total weight of all such paths of length \(n\) ending
at height \(k\).  Thus \(w_{0,0}=1\), and there are no paths of length
zero ending at a positive height.  We also set
\[
   w_{n,k}=0\quad(k<0).
\]
Looking at the last step gives
\begin{equation}
\label{eq:motzkin-rec}
   w_{n+1,k}
   =
   \alpha_{k-1}w_{n,k-1}
   +
   \gamma_k w_{n,k}
   +
   \beta_k w_{n,k+1},
   \qquad k\ge0,
\end{equation}
where the first term is omitted when \(k=0\).  If the weights are
nonnegative integers, \(w_{n,k}\) counts coloured Motzkin paths, where the
weight of a step is the number of allowed colours for that step.  For
general nonnegative real weights, \(w_{n,k}\) is the corresponding total
path weight.

The boundary numbers
\[
   W_n:=w_{n,0}
\]
count weighted paths which return to height \(0\); we call \((W_n)_{n\ge0}\)
the \emph{return sequence}.  The word ``return'' means only that the
terminal height is \(0\); no first-return condition is imposed.  Thus the
main enumerative objects are both the return sequence and the full
height-refined triangle \((w_{n,k})_{n,k\ge0}\).

We package the rows into polynomials
\[
   P_n(x)=\sum_{k\ge0}w_{n,k}x^k,
\]
and use the generating function which is polynomial in \(x\) row by row and
exponential in \(t\):
\begin{equation}
\label{eq:motzkin-egf}
   w(x,t)
   =
   \sum_{n\ge0}P_n(x)\frac{t^n}{n!}
   =
   \sum_{n\ge0}\sum_{k\ge0}
   w_{n,k}x^k\frac{t^n}{n!}.
\end{equation}
The coefficient of \(x^k t^n/n!\) in \(w(x,t)\) is exactly the total weight
of paths of length \(n\) ending at height \(k\); for the affine weights
considered below, the reason for the exponential scale in \(t\) will be
visible in the step rule.  For affine weights, the series
\eqref{eq:motzkin-egf} is analytic near \(t=0\); a direct path bound is given
in Appendix~\ref{app:continuous-kernel-details},
Lemma~\ref{lem:app-egf-analytic}.  The return generating function is
\[
   W(t)=w(0,t)=\sum_{n\ge0}W_n\frac{t^n}{n!}
   =
   \sum_{n\ge0}w_{n,0}\frac{t^n}{n!}.
\]

We shall study affine height-dependent weights
\[
   \alpha_k=A k+\alpha_0,
   \qquad
   \beta_k=C k+\beta_0,
   \qquad
   \gamma_k=B k+\gamma_0.
\]
For a genuine nonnegative weighted-path model on all heights, it is enough
to assume
\[
   A,B,C,\alpha_0,\beta_0,\gamma_0\ge0.
\]
If the affine parameters are real but some step weights are negative, the
same definitions give a signed weighted-path model.  The row recurrences and
coefficient identities are, more generally, formal identities over any
coefficient field of characteristic zero.  Unless explicitly stated
otherwise, the characteristic arguments in
Sections~\ref{sec:continuous-kernel} and~\ref{sec:integer-shift} use real
parameters with \(C>0\) and \(\beta_0>C\).  The special-function
identities in Appendix~\ref{app:continuous-kernel-details} are first proved
on an analytic parameter domain and are then interpreted, where stated, by
meromorphic continuation or coefficientwise as formal Laurent-series
identities.

Put
\[
   Q(x)=Ax^2+Bx+C,
   \qquad
   \delta=\beta_0-C,
\]
and let
\[
   \Delta_0 p(x)=\frac{p(x)-p(0)}{x}
\]
be the divided difference at the boundary.  Multiplying
\eqref{eq:motzkin-rec} by \(x^k\), summing over \(k\ge0\), and using the
affine weights gives the row recurrence
\begin{equation}
\label{eq:polynomial-recurrence}
   P_{n+1}(x)
   =
   Q(x)P_n'(x)
   +
   (\alpha_0x+\gamma_0)P_n(x)
   +
   \delta\,\Delta_0P_n(x).
\end{equation}
The affine structure of the weights is exactly what reduces the interior
part of the transition \(P_n\mapsto P_{n+1}\) to a first-order differential
expression in \(x\), the only remaining nonlocal contribution being the
boundary divided difference.  This is why the exponential scale in \(t\) is
the natural one.  Multiplying
\eqref{eq:polynomial-recurrence} by \(t^n/n!\) and summing over \(n\) gives
\begin{equation}
\label{eq:motzkin-pde}
   w_t
   =
   Q(x)w_x
   +
   (\alpha_0x+\gamma_0)w
   +
   \delta\,\frac{w(x,t)-W(t)}{x},
   \qquad
   w(x,0)=1.
\end{equation}
At \(x=0\) the quotient is understood by its removable value \(w_x(0,t)\).
Thus \eqref{eq:motzkin-pde} is a first-order equation in \((x,t)\), amenable
to characteristics, but with a possible boundary coupling.

The meaning of \(\delta\) is simple.  For a down-step leaving height
\(j\ge1\) and arriving at height \(j-1\), the weight is
\[
   \beta_{j-1}=Cj+\delta.
\]
The part \(Cj\), proportional to the departure height, is exactly the
contribution represented by \(CP_n'(x)\) in
\eqref{eq:polynomial-recurrence}.  The constant remainder
\(\delta=\beta_0-C\) is the same at every height, and summing these
remainders over all departure heights \(j\ge1\) gives
\[
   \delta\sum_{j\ge1}w_{n,j}x^{j-1}
   =
   \delta\,\frac{P_n(x)-P_n(0)}{x}:
\]
this is the boundary obstruction in its combinatorial form.

We call the affine row evolution \emph{local} if it is represented by a
fixed first-order differential expression
\[
   P_{n+1}(x)=p_1(x)P_n'(x)+p_0(x)P_n(x),
\]
with polynomial coefficients \(p_0,p_1\), and contains no separate
evaluation of \(P_n\) at \(x=0\).  Equivalently, the exponential
generating-function equation contains no separately evaluated boundary
series \(W(t)=w(0,t)\).

\begin{proposition}[Locality of the affine row rule]
The affine generating-function evolution is local if and only if
\begin{equation}
\label{eq:balanced-condition}
   \beta_0=C,
   \qquad\text{equivalently}\qquad
   \delta=0.
\end{equation}
\end{proposition}

\begin{proof}
The divided-difference term is the only term in
\eqref{eq:polynomial-recurrence} or \eqref{eq:motzkin-pde} that contains a
separate evaluation at the boundary, and its coefficient is
\(\delta=\beta_0-C\).
\end{proof}

We call \eqref{eq:balanced-condition} the \emph{balanced} condition.  The
balanced case is treated in detail in \cite{OmelchenkoBalanced}.  The
present paper studies the boundary problem that remains when the balance is
broken.

When \(C>0\), we normalise the imbalance by writing
\[
   \nu=\frac{\delta}{C}=\frac{\beta_0-C}{C},
   \qquad
   \beta_{j-1}=C(j+\nu),
   \qquad j\ge1.
\]
We call \(\nu\) the \emph{boundary index}.  Thus \(\nu\) is the normalised
excess in every down-step weight over its balanced part \(Cj\).  The value
\(\nu=0\) is precisely the balanced case.  Integer values
\(\nu=m\in\mathbb Z_{>0}\) will later be interpreted as finite shifts of
the height origin.  Noninteger values carry no literal height shift, but
they still define genuine weighted models whenever the resulting step
weights are nonnegative.

\begin{remark}[Elementary boundary degeneracies]
If \(\alpha_0=0\), no nonzero-weight path can leave height \(0\).  If
\(\beta_0=0\), a path that has left height \(0\) cannot return with nonzero
weight.  In either case, only strings of level-steps at height \(0\)
contribute to the return sequence, and
\[
   W(t)=e^{\gamma_0t}.
\]
These elementary cases require no boundary analysis and will not be
discussed further.
\end{remark}

If \(\gamma_k=0\) for every \(k\), level steps are forbidden and the
Motzkin model becomes a weighted Dyck model.  Figure~\ref{fig:dyck-sec1}
shows the general weighted Dyck triangle in panel~(a) and a concrete affine
example with positive boundary index in panel~(b).  The latter example is analysed in
Example~\ref{ex:ck-nu2-boundary-equation}.

\begin{figure}[H]
\centering
\includegraphics[width=.47\textwidth]{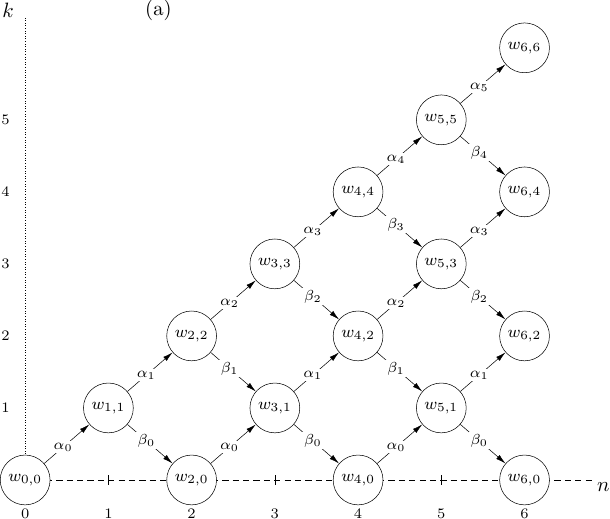}\hfill
\includegraphics[width=.47\textwidth]{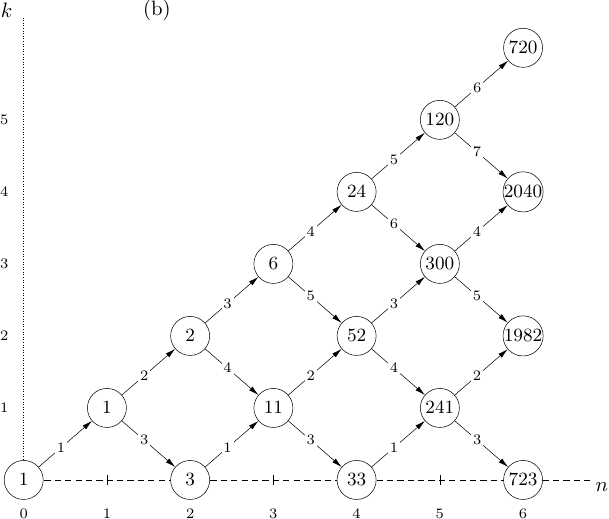}
\caption{Weighted Dyck triangles.  Panel (a): the general Dyck triangle;
each up-step carries the weight \(\alpha_k\) of its departure height, and
each down-step the weight \(\beta_k\) of its arrival height.  Panel (b):
the boundary-index-two example \(\alpha_k=k+1\), \(\beta_k=k+3\), analysed
in Example~\ref{ex:ck-nu2-boundary-equation}.  The vertex labels are the corresponding
total path weights.  Entries with \(n+k\) odd vanish because every step
changes the height by \(\pm1\).}
\label{fig:dyck-sec1}
\end{figure}

Unless explicitly stated otherwise, throughout
Sections~\ref{sec:continuous-kernel}--\ref{sec:integer-shift} we assume
\begin{equation}
\label{eq:positive-index-assumptions}
   C>0,
   \qquad
   \beta_0>C,
   \qquad
   \nu=\frac{\beta_0-C}{C}>0.
\end{equation}
Under these assumptions, \(\nu\) measures the positive boundary excess in
the down-step weights.

Although \(W(t)=w(0,t)\), the characteristic
representation of \eqref{eq:motzkin-pde} contains \(W\) as an a priori
unknown return series.  The boundary problem is to determine this series
from the affine path model itself and then reconstruct the full
height-refined generating function \(w(x,t)\).  The next section solves this boundary problem.


\section{The return equation and height reconstruction}
\label{sec:continuous-kernel}

The recurrence \eqref{eq:motzkin-rec} computes the whole affine Motzkin
triangle row by row, even when only the return numbers
\(W_n=w_{n,0}\) are wanted.  We now derive a closed one-dimensional
equation for their exponential generating function \(W(t)\), and then use
the same calculation to reconstruct the full height-refined generating
function \(w(x,t)\).

Throughout this section the standing assumptions
\eqref{eq:positive-index-assumptions} are in force, so that
\[
   \delta=\beta_0-C=\nu C>0.
\]
We apply the method of characteristics to \eqref{eq:motzkin-pde} in its
elementary form: we differentiate \(w\) along a curve and choose that curve
so that the derivative \(w_x\) disappears.\footnote{For standard
references, see \cite[Section~3.2]{Evans2010} and
\cite{CoddingtonLevinson1955}.}  At this stage \(W(t)\) is left
unspecified; the calculation will produce the equation that it must satisfy.

Fix a sufficiently small \(T>0\), and consider a curve starting at the
boundary point \((0,T)\), written as
\[
   \tau\longmapsto \bigl(X(\tau),T-\tau\bigr),
   \qquad X(0)=0.
\]
Put
\[
   U_T(\tau):=w(X(\tau),T-\tau).
\]
The chain rule gives
\[
   U_T'(\tau)
   =X'(\tau)w_x(X(\tau),T-\tau)
    -w_t(X(\tau),T-\tau).
\]
Substituting \eqref{eq:motzkin-pde} yields
\begin{align*}
   U_T'(\tau)
   &={}
   \bigl(X'(\tau)-Q(X(\tau))\bigr)
      w_x(X(\tau),T-\tau) \\
   &\quad
   -(\alpha_0X(\tau)+\gamma_0)U_T(\tau)
   -\frac{\delta}{X(\tau)}
      \bigl(U_T(\tau)-W(T-\tau)\bigr).
\end{align*}
We choose \(X\) so that the term containing \(w_x\) vanishes:
\begin{equation}
\label{eq:ck-boundary-flow}
   X'(\tau)=Q(X(\tau)),
   \qquad
   X(0)=0.
\end{equation}
Thus \(\tau\mapsto(X(\tau),T-\tau)\) is the characteristic through
\((0,T)\), followed backwards from the boundary.  Along this characteristic,
\(U_T\) satisfies the linear ordinary differential equation
\begin{equation}
\label{eq:ck-characteristic-ode}
   U_T'(\tau)
   +
   \left(
      \alpha_0X(\tau)+\gamma_0+
      \frac{\delta}{X(\tau)}
   \right)U_T(\tau)
   =
   \frac{\delta}{X(\tau)}W(T-\tau).
\end{equation}

For a linear equation \(U'+aU=b\), an integrating factor \(F\) is determined
by \(F'/F=a\).  Here
\[
   a(\tau)
   =\alpha_0X(\tau)+\gamma_0+
      \frac{\delta}{X(\tau)}.
\]
The affine equation \eqref{eq:ck-boundary-flow} gives
\[
   \frac{X'}{X}=AX+B+\frac{C}{X}.
\]
Since \(\delta=\nu C\),
\[
   \nu\frac{X'}{X}
   =\nu AX+\nu B+\frac{\delta}{X},
\]
and therefore
\begin{equation}
\label{eq:ck-coeff-split}
   \alpha_0X+\gamma_0+\frac{\delta}{X}
   =
   \nu\frac{X'}{X}
   +(\alpha_0-\nu A)X
   +(\gamma_0-\nu B).
\end{equation}
Since the first term in \eqref{eq:ck-coeff-split} integrates to
\(\log X^\nu\), the corresponding integrating factor is
\begin{equation}
\label{eq:ck-F-def}
   F_\nu(\tau)
   :=
   X(\tau)^\nu
   \exp\left(
      \int_0^\tau
      \bigl[(\alpha_0-\nu A)X(r)+\gamma_0-\nu B\bigr]dr
   \right).
\end{equation}
For sufficiently small \(\tau>0\), one has \(X(\tau)>0\), and the power
\(X(\tau)^\nu\) is the positive real power.  By construction,
\[
   \frac{F_\nu'(\tau)}{F_\nu(\tau)}
   =
   \alpha_0X(\tau)+\gamma_0+
   \frac{\delta}{X(\tau)}.
\]
Multiplying \eqref{eq:ck-characteristic-ode} by \(F_\nu(\tau)\) now gives
\[
   \frac{d}{d\tau}
   \bigl[F_\nu(\tau)U_T(\tau)\bigr]
   =
   \frac{\delta F_\nu(\tau)}{X(\tau)}W(T-\tau).
\]
For brevity, set
\[
   L_\nu(\tau)
   :=
   \frac{\delta F_\nu(\tau)}{X(\tau)}.
\]

Since \(X(0)=0\) and \(X'(0)=Q(0)=C\),
\[
   X(\tau)=C\tau+O(\tau^2).
\]
It follows from \eqref{eq:ck-F-def} that
\begin{equation}
\label{eq:ck-FL-leading}
   F_\nu(\tau)
   =C^\nu\tau^\nu\bigl(1+O(\tau)\bigr),
   \qquad
   L_\nu(\tau)
   =\nu C^\nu\tau^{\nu-1}\bigl(1+O(\tau)\bigr).
\end{equation}
Thus \(F_\nu(\tau)\to0\) as \(\tau\downarrow0\), and \(L_\nu\) is
integrable at the origin because \(\nu>0\).

\begin{theor_eng}[The return equation]
\label{thm:boundary-selection-combinatorial}
The exponential generating function for weighted returns is the unique
function \(W\), analytic near the origin and satisfying \(W(0)=1\), for
which
\begin{equation}
\label{eq:ck-volterra-final}
   F_\nu(T)
   =
   \int_0^T L_\nu(v)W(T-v)\,dv
\end{equation}
holds for all sufficiently small \(T>0\).  Thus the affine step weights
determine the return sequence without first computing the entries at
positive terminal heights.
\end{theor_eng}

\begin{proof}
For \(0<\varepsilon<\tau\), integration of the characteristic equation
gives
\[
   F_\nu(\tau)U_T(\tau)
   -F_\nu(\varepsilon)U_T(\varepsilon)
   =
   \int_\varepsilon^\tau
      L_\nu(v)W(T-v)\,dv.
\]
As \(\varepsilon\downarrow0\),
\[
   U_T(\varepsilon)
   =w(X(\varepsilon),T-\varepsilon)
   \longrightarrow w(0,T)=W(T),
\]
whereas \eqref{eq:ck-FL-leading} gives
\(F_\nu(\varepsilon)U_T(\varepsilon)\to0\).  Hence
\begin{equation}
\label{eq:ck-boundary-flow-identity}
   F_\nu(\tau)w(X(\tau),T-\tau)
   =
   \int_0^\tau L_\nu(v)W(T-v)\,dv.
\end{equation}
Setting \(\tau=T\), the time coordinate becomes zero and
\(w(X(T),0)=1\), which gives \eqref{eq:ck-volterra-final}.  The path
generating function supplies existence.  Uniqueness among analytic
solutions follows from Lemma~\ref{lem:app-abel-regularisation}; the required
analyticity estimate is Lemma~\ref{lem:app-egf-analytic}.
\end{proof}

The characteristic identity \eqref{eq:ck-boundary-flow-identity} will also
give the reconstruction formula below.

Corollary~\ref{cor:ck-beta-recursion} gives a constructive form of this
uniqueness: it determines \(W_0,W_1,\ldots\) successively.  Thus
\eqref{eq:ck-volterra-final} is not an additional boundary condition; it is
the one-dimensional equation imposed on the return series by the affine path
model.

To recover an interior value, let \(\rho\) be the local inverse of \(X\).
Separating variables in \eqref{eq:ck-boundary-flow} gives
\[
   \rho(x)=X^{-1}(x)=\int_0^x\frac{dy}{Q(y)}.
\]

\begin{theor_eng}[Height-refined reconstruction]
For sufficiently small \(x>0\) and \(t\) close to the origin, the full
height-refined generating function is recovered from the return series by
\begin{equation}
\label{eq:ck-reconstruction-final}
   w(x,t)
   =
   \frac{1}{F_\nu(\rho(x))}
   \int_0^{\rho(x)}
      L_\nu(v)W(t+\rho(x)-v)\,dv.
\end{equation}
The right-hand side extends continuously to \(x=0\), where it equals
\(W(t)\).  Consequently, the return sequence together with the affine step
weights determines every coefficient \(w_{n,k}\).
\end{theor_eng}

\begin{proof}
Let \(r=\rho(x)\), so that \(X(r)=x\), and put \(T=t+r\).  Taking
\(\tau=r\) in \eqref{eq:ck-boundary-flow-identity} gives
\[
   F_\nu(r)w(x,t)
   =
   \int_0^r L_\nu(v)W(t+r-v)\,dv,
\]
which is \eqref{eq:ck-reconstruction-final} for \(x\ne0\).

It remains to verify the boundary value directly.  As \(r\downarrow0\),
uniformly for \(0\le v\le r\),
\[
   L_\nu(v)=\nu C^\nu v^{\nu-1}\bigl(1+O(r)\bigr),
   \qquad
   W(t+r-v)=W(t)+O(r).
\]
Therefore
\[
   \int_0^r L_\nu(v)W(t+r-v)\,dv
   =C^\nu r^\nu W(t)+O(r^{\nu+1}).
\]
Since \(F_\nu(r)=C^\nu r^\nu(1+O(r))\), the quotient in
\eqref{eq:ck-reconstruction-final} equals \(W(t)+O(r)\).  It thus extends
continuously to \(x=0\), with value \(W(t)\).
\end{proof}

The method of characteristics and the integrating-factor calculation above
are classical.  The terminology in the title refers to the cancellation at
the boundary that determines \(W\).  In the classical linear kernel method,
an equation with one catalytic variable has the form
\[
   K(t,u)\,\mathcal F(t,u)
   =R(F_1,\ldots,F_k,t,u),
\]
and substitution of an admissible root \(u=U(t)\) of \(K\) eliminates the
unknown bivariate series and leaves an equation for the boundary series; see
\cite[Section~3.1]{BousquetMelouJehanne2006}.  Here there is no algebraic
kernel root.  Instead, \(F_\nu(\tau)\to0\) at the boundary end
\(\tau=0\), and the term containing the unknown value \(W(T)\) disappears
from \eqref{eq:ck-boundary-flow-identity}.  Evaluating the same identity on
the initial line \(T-\tau=0\) then gives the return equation
\eqref{eq:ck-volterra-final}.  We call this boundary-cancellation step the
\emph{differential kernel method}.

\begin{corollary}[A triangular recursion for the return weights]
\label{cor:ck-beta-recursion}
Write
\[
   F_\nu(t)=t^\nu\sum_{j\ge0}h_jt^j,
   \qquad
   L_\nu(t)=t^{\nu-1}\sum_{j\ge0}\ell_jt^j,
\]
where \(h_0=C^\nu\) and \(\ell_0=\nu C^\nu\).  Then \(W_0=1\), and for
\(N\ge1\),
\begin{equation}
\label{eq:ck-beta-recursion}
   W_N
   =
   \frac{N!}{\ell_0\,\mathrm B(\nu,N+1)}
   \left[
      h_N-
      \sum_{j=1}^{N}
      \ell_j\frac{W_{N-j}}{(N-j)!}
      \mathrm B(\nu+j,N-j+1)
   \right].
\end{equation}
Here \(\mathrm B(\cdot,\cdot)\) denotes the beta function, upright \(\mathrm B\)
being reserved for it throughout, in contrast with the affine level-weight
parameter \(B\).  Since \(\nu>0\), all beta factors in the formula are finite
and positive.
\end{corollary}

\begin{proof}
Insert the displayed expansions and
\(W(t)=\sum_{n\ge0}W_nt^n/n!\) into
\eqref{eq:ck-volterra-final}.  The coefficient of \(T^{\nu+N}\) is
\[
   h_N
   =
   \sum_{j=0}^{N}
      \ell_j\frac{W_{N-j}}{(N-j)!}
      \mathrm B(\nu+j,N-j+1).
\]
The term \(j=0\) contributes
\(\ell_0W_N\,\mathrm B(\nu,N+1)/N!\).  Since
\(\ell_0=\nu C^\nu>0\) and \(\mathrm B(\nu,N+1)>0\), this coefficient is nonzero,
and solving for \(W_N\) gives \eqref{eq:ck-beta-recursion}.
\end{proof}

\begin{example}[The boundary-index-two Dyck triangle]
\label{ex:ck-nu2-boundary-equation}
Consider the weights
\[
   \alpha_k=k+1,
   \qquad
   \beta_k=k+3,
   \qquad
   \gamma_k=0,
\]
shown in Figure~\ref{fig:dyck-sec1}(b).  Here
\[
   Q(x)=1+x^2,
   \qquad
   \nu=2,
   \qquad
   X(t)=\tan t,
   \qquad
   \rho(x)=\arctan x.
\]
Formula \eqref{eq:ck-F-def} gives
\[
   F_2(t)=\frac{\sin^2t}{\cos t},
   \qquad
   L_2(t)=2\sin t.
\]
The return equation becomes
\begin{equation}
\label{eq:ck-nu2-return-equation}
   \frac{\sin^2T}{\cos T}
   =2\int_0^T\sin v\,W(T-v)\,dv.
\end{equation}
Let
\[
   K(T):=\int_0^T\sin v\,W(T-v)\,dv.
\]
After the change of variable \(u=T-v\), differentiation gives
\(K''+K=W\).  Since \eqref{eq:ck-nu2-return-equation} says
\(2K=\sin^2T/\cos T\), direct differentiation yields
\[
   W(T)=\sec^3T.
\]

The reconstruction formula, with \(r=\rho(x)=\arctan x\), now gives
\[
   w(x,t)
   =
   \frac{2}{F_2(r)}
   \int_0^r\sin v\,\sec^3(t+r-v)\,dv.
\]
With the substitution \(s=t+r-v\),
\begin{align*}
   2\int_0^r\sin v\,\sec^3(t+r-v)\,dv
   &=2\int_t^{t+r}\sin(t+r-s)\,\sec^3s\,ds \\
   &=\left[
      2\sin(t+r)\tan s
      -\cos(t+r)\tan^2s
   \right]_{s=t}^{t+r}.
\end{align*}
Using \(\tan r=x\) and the tangent addition formula, this simplifies to
\[
   2\int_0^r\sin v\,\sec^3(t+r-v)\,dv
   =
   \frac{\sin^2r}{\cos r}
   \frac{\sec^3t}{1-\tan r\tan t},
\]
so that
\[
   w(x,t)=\frac{\sec^3t}{1-x\tan t}.
\]
Thus
\[
   W_n=n!\,[t^n]\sec^3t,
   \qquad
   w_{n,k}=n!\,[t^n]\sec^3t\,\tan^k t,
\]
and the first nonzero return weights are
\[
   1,\ 3,\ 33,\ 723,\ 25953,\ldots
\]
\end{example}

The gamma-normalised formal ordinary-series quotient, its generic hypergeometric
evaluation, and computational remarks are collected in
Appendix~\ref{app:continuous-kernel-details}.  We next turn to the direct
path interpretation of the return equation when \(\nu\) is a positive
integer.


\section{Integer boundary index and the shifted floor}
\label{sec:integer-shift}

Assume that the boundary index is a positive integer,
\begin{equation}
\label{eq:integer-shift-assumption}
   \nu=m\in\mathbb Z_{>0},
   \qquad
   \beta_0=(m+1)C.
\end{equation}
Then
\[
   \beta_k=C(k+m+1),
\]
and the occurrence of \(k+m\) suggests translating the height scale.  Put
\[
   j=k+m.
\]
If the new height is still restricted to \(j\ge m\), this is only a
relabelling of the original paths.  The row polynomials and the bivariate
generating function become
\[
   v_n(x):=x^mP_n(x),
   \qquad
   v(x,t):=x^mw(x,t).
\]
Thus \(v\) describes the original Motzkin triangle shifted upward by
\(m\) levels: its paths start at height \(m\) and never go below that
height.

We now extend the shifted model to the whole half-line \(j\ge0\).  The
levels \(j\ge m\) form the \emph{visible region}, and
\[
   0\le j<m
\]
form the \emph{virtual strip}.  At a visible height \(j=k+m\), the original
weights become
\[
\begin{aligned}
   \widetilde\alpha_j
   &=\alpha_{j-m}=Aj+\alpha_0-mA,\\
   \widetilde\beta_j
   &=\beta_{j-m}=C(j+1),\\
   \widetilde\gamma_j
   &=\gamma_{j-m}=Bj+\gamma_0-mB.
\end{aligned}
\]
We use the affine expressions on the right-hand side to define the
auxiliary weights for every \(j\ge0\).  In particular, the down-step rule
has the form \(\widetilde\beta_j=C(j+1)\), so the auxiliary row evolution
is local and contains no boundary divided difference.  The affine
expressions, extended into the virtual strip, may take zero or negative
values there; the auxiliary model is therefore understood as a formal
weighted-path model whenever necessary.  At visible levels \(j\ge m\), the auxiliary step
weights agree with the translated original weights.  Hence auxiliary paths
which remain in \(j\ge m\) are precisely the shifted original paths.  The
unrestricted auxiliary totals need not agree with the original visible
coefficients, since a path may enter the virtual strip and later return; the
first-entry correction below measures exactly this discrepancy.

Let \(a_{n,j}\) be the total weight of auxiliary shifted paths of length
\(n\) which start at height \(m\), remain in \(j\ge0\), and end at height
\(j\).  Put
\[
   a_n(x):=\sum_{j\ge0}a_{n,j}x^j,
   \qquad
   a(x,t):=\sum_{n\ge0}a_n(x)\frac{t^n}{n!}.
\]
Set
\[
   \mathcal P_m f
   :=Q(x)f'(x)
     +\bigl((\alpha_0-mA)x+\gamma_0-mB\bigr)f(x).
\]
Then
\[
   a_t=\mathcal P_m a,
   \qquad
   a(x,0)=x^m.
\]
Thus \(a(x,t)\) describes the freely evolved shifted triangle, in which
paths may enter and leave the virtual strip.

The advantage of the integer shift is also visible directly in the equation
for \(v=x^mw\).  Since
\[
   x^mw_x=v_x-\frac{m}{x}v
\]
and \(\delta=mC\), multiplication of \eqref{eq:motzkin-pde} by \(x^m\)
gives
\begin{align*}
   v_t
   &={}Q(x)\left(v_x-\frac{m}{x}v\right)
      +(\alpha_0x+\gamma_0)v
      +mC\left(\frac{v}{x}-x^{m-1}W(t)\right)\\
   &={}\mathcal P_m v-mCx^{m-1}W(t).
\end{align*}
Hence
\begin{equation}
\label{eq:integer-shifted-v-pde}
   v_t=\mathcal P_m v-mCx^{m-1}W(t),
   \qquad
   v(x,0)=x^m.
\end{equation}
The divided difference has disappeared.  Its effect is now the regular
source supported at the top virtual level \(j=m-1\): it removes paths which
step from the visible floor \(m\) into the virtual strip.

To describe this correction globally, let \(b_{n,j}\) be the total weight
of the same auxiliary shifted paths, now starting at height \(m-1\), and
put
\[
   b(x,t):=\sum_{n,j\ge0}b_{n,j}x^j\frac{t^n}{n!}.
\]
It satisfies the same local equation,
\[
   b_t=\mathcal P_m b,
   \qquad
   b(x,0)=x^{m-1}.
\]

\begin{proposition}[First-entry reconstruction in the shifted model]
\label{prop:integer-first-entry}
For \(n,j\ge0\), the auxiliary coefficients satisfy
\[
   a_{n,j}
   =\mathbf 1_{\{j\ge m\}}w_{n,j-m}
    +mC\!\sum_{\substack{r,q\ge0\\r+q=n-1}}
       W_r b_{q,j},
\]
with the sum empty when \(n=0\).  Equivalently, the free auxiliary triangle
decomposes as
\begin{equation}
\label{eq:integer-first-entry-reconstruction}
   a(x,t)
   =x^mw(x,t)
    +mC\int_0^t b(x,s)W(t-s)\,ds.
\end{equation}
Thus
\begin{equation}
\label{eq:integer-duhamel-reconstruction}
   x^mw(x,t)
   =a(x,t)
    -mC\int_0^t b(x,s)W(t-s)\,ds.
\end{equation}
\end{proposition}

\begin{proof}
Every auxiliary path counted by \(a(x,t)\) belongs to exactly one of two
classes.  If it never enters the virtual strip, shifting all heights down by
\(m\) gives an original Motzkin path; these paths are counted by
\(x^mw(x,t)\).

Otherwise, mark the first step from \(m\) to \(m-1\).  Before the marked
step, the path starts and ends at height \(m\) and remains in the visible
region; after shifting down by \(m\), this initial segment is an original
return path counted by \(W\).  The marked step has weight
\[
   \widetilde\beta_{m-1}=mC,
\]
and the remaining segment is an arbitrary auxiliary path starting at
height \(m-1\), counted by \(b\).  The decomposition is unique and gives the displayed coefficient identity.
Passing to exponential generating functions uses
\[
   \int_0^t\frac{s^q}{q!}\frac{(t-s)^r}{r!}\,ds
   =\frac{t^{q+r+1}}{(q+r+1)!},
\]
which yields \eqref{eq:integer-first-entry-reconstruction}.
\end{proof}

This is the integer-index form of the reconstruction theorem of
Section~\ref{sec:continuous-kernel}.  Its additional content is that the two
functions \(a\) and \(b\) are themselves generating functions of completely
local weighted-path triangles.  Analytically,
\eqref{eq:integer-duhamel-reconstruction} is Duhamel's formula for the
inhomogeneous equation \eqref{eq:integer-shifted-v-pde}; the direct
verification is recorded in Appendix~\ref{app:integer-compression}.

For a polynomial or formal power series in \(x\), write
\[
   [f(x)]_{\ge m}
   :=f(x)-\sum_{j=0}^{m-1}[x^j]f(x)\,x^j.
\]
The coefficient-level form of \eqref{eq:integer-shifted-v-pde} is the
following exact deletion rule.

\begin{corollary}[Row-by-row shifted-floor rule]
For every \(n\ge0\),
\begin{equation}
\label{eq:integer-row-by-row-rule}
   x^mP_{n+1}(x)
   =\left[\mathcal P_m\bigl(x^mP_n(x)\bigr)\right]_{\ge m}.
\end{equation}
Both sides are written on the shifted scale.  The truncation deletes the
endpoints in the virtual strip; division by \(x^m\) is the separate final
operation that returns to the original height scale.
\end{corollary}

\begin{proof}
Taking the coefficient of \(t^n/n!\) in
\eqref{eq:integer-shifted-v-pde} gives
\[
   x^mP_{n+1}(x)
   =\mathcal P_m\bigl(x^mP_n(x)\bigr)-mCW_nx^{m-1}.
\]
The last term is the only term of degree below \(m\), and truncation removes
it.
\end{proof}

Combinatorially, the deleted monomial
\[
   mC\,W_nx^{m-1}=mC\,P_n(0)x^{m-1}
\]
counts paths which are at the original floor after \(n\) steps, hence at
shifted height \(m\), and then take the down-step to the top level
\(m-1\) of the virtual strip.

Setting \(x=0\) in \eqref{eq:integer-first-entry-reconstruction} gives the
one-dimensional return equation
\begin{equation}
\label{eq:integer-first-entry-convolution}
   a(0,t)
   =mC\int_0^t b(0,s)W(t-s)\,ds.
\end{equation}
Appendix~\ref{app:integer-compression} evaluates the two local boundary
series as
\[
\begin{aligned}
   a(0,t)
   &=X(t)^m
     \exp\left(
        \int_0^t
        \bigl[(\alpha_0-mA)X(u)+\gamma_0-mB\bigr]du
     \right),\\
   b(0,t)
   &=X(t)^{m-1}
     \exp\left(
        \int_0^t
        \bigl[(\alpha_0-mA)X(u)+\gamma_0-mB\bigr]du
     \right).
\end{aligned}
\]
Comparison with the functions of Section~\ref{sec:continuous-kernel} gives
\[
   F_m(t)=a(0,t),
   \qquad
   L_m(t)=mC\,b(0,t).
\]
Thus \eqref{eq:integer-first-entry-convolution} is exactly the return equation
of Section~\ref{sec:continuous-kernel}, and the general characteristic
reconstruction is realised at integer index by two local shifted path
triangles.

A path from height \(m-1\) to \(0\) needs at least \(m-1\) steps.  Its
unique shortest realisation consists entirely of down-steps, and hence
\[
   b(0,t)=C^{m-1}t^{m-1}+O(t^m).
\]
Because \(m\) is an integer, the functions
\(F_m(t)=a(0,t)\) and \(L_m(t)=mC\,b(0,t)\) are analytic at the origin.
Moreover, \(L_m\) has a zero of order \(m-1\), with
\[
   L_m^{(q)}(0)=0\quad(0\le q\le m-2),
   \qquad
   L_m^{(m-1)}(0)=m!C^m.
\]
Thus ordinary differentiation of the convolution is legitimate here; no
fractional Abel regularisation is needed.  After \(m\) differentiations,
\eqref{eq:integer-first-entry-convolution} becomes a Volterra equation of
the second kind.

\begin{corollary}[Second-kind return equation and coefficient recursion]
The return series satisfies
\begin{equation}
\label{eq:integer-second-kind-volterra}
   W(t)
   =\frac{1}{m!C^m}
    \left[
       \partial_t^m a(0,t)
       -mC\int_0^t
          \partial_t^m b(0,t-s)W(s)\,ds
    \right].
\end{equation}
Equivalently,
\[
   W(t)
   =\frac{1}{m!C^m}
    \left[
       F_m^{(m)}(t)
       -\int_0^tL_m^{(m)}(t-s)W(s)\,ds
    \right].
\]
If
\[
   F_m(t)=\sum_{n\ge0}f_n\frac{t^n}{n!},
   \qquad
   L_m(t)=\sum_{n\ge0}\lambda_n\frac{t^n}{n!},
\]
then \(\lambda_{m-1}=m!C^m\), \(W_0=1\), and for every \(N\ge1\),
\begin{equation}
\label{eq:integer-return-coefficient-recursion}
   W_N
   =\frac{1}{m!C^m}
    \left[
       f_{N+m}
       -\sum_{q=1}^{N}
          \lambda_{m-1+q}W_{N-q}
    \right].
\end{equation}
\end{corollary}

The proof by differentiation, including the endpoint term
\(m!C^mW(t)\), is given in Appendix~\ref{app:integer-compression}.  The
recursion is triangular and computes the original return sequence from the
two boundary series \(a(0,t)\) and \(b(0,t)\) of the local shifted model.
For generic affine parameters, integer \(m\) also gives a finite
special-function reduction: the reciprocal return series is a logarithmic
derivative of a Gauss function and can be expressed through finitely many
ordinary derivatives of one incomplete beta function; see
Corollary~\ref{cor:integer-incomplete-beta-reduction}.

We now ask when the correction term in
\eqref{eq:integer-duhamel-reconstruction} has no visible coefficients.  A
path can return from the virtual strip to the visible region only through
the step
\[
   m-1\longrightarrow m,
\]
whose weight is
\[
   \widetilde\alpha_{m-1}
   =A(m-1)+\alpha_0-mA
   =\alpha_0-A.
\]

\begin{proposition}[Criterion for one final deletion]
\label{prop:one-final-deletion-criterion}
Under \eqref{eq:integer-shift-assumption}, the following statements are
equivalent:
\begin{enumerate}
\item \(\alpha_0=A\);
\item \([a_n(x)]_{\ge m}=x^mP_n(x)\) for every \(n\ge0\);
\item \(a_{n,k+m}=w_{n,k}\) for every \(n,k\ge0\).
\end{enumerate}
When these conditions hold,
\[
   x^mw(x,t)=[a(x,t)]_{\ge m},
\]
and
\[
   W(t)=[x^m]a(x,t)
       =\frac1{m!}\,\partial_x^m a(x,t)\big|_{x=0}.
\]
If \(\alpha_0\ne A\), the discrepancy is already visible after two steps:
\[
   a_{2,m}-w_{2,0}=mC(\alpha_0-A).
\]
\end{proposition}

\begin{proof}
Suppose first that \(\alpha_0=A\).  Then the bridge \(m-1\to m\) has
weight zero, so an auxiliary path which enters the virtual strip cannot make
a nonzero contribution to a visible endpoint.  The visible auxiliary paths
are therefore exactly the shifted images of the original paths, proving
statements 2 and 3.

Conversely, any two-step path from \(m\) back to \(m\) either stays in the
visible region or dips once to \(m-1\) and returns.  The latter is the only
such path which enters the virtual strip, and its weight is
\[
   \widetilde\beta_{m-1}\widetilde\alpha_{m-1}
   =mC(\alpha_0-A).
\]
Hence \(a_{2,m}=w_{2,0}+mC(\alpha_0-A)\).  Since \(mC>0\), equality of the
visible coefficients forces \(\alpha_0=A\).  Statements 2 and 3 are
equivalent by coefficient extraction, and summing over \(n\) gives the two
generating-function identities.
\end{proof}

We finish with two examples having one virtual level.

\begin{example}[A closed bridge]
Take
\[
   \alpha_k=k+1,
   \qquad
   \beta_k=k+2,
   \qquad
   \gamma_k=k+1.
\]
Here \(m=1\) and \(\alpha_0=A=1\), so the bridge \(0\to1\) is closed and
Proposition~\ref{prop:one-final-deletion-criterion} applies.  The first rows
of the original and freely shifted triangles are shown in
Figure~\ref{fig:cascade-basement}.  The auxiliary function satisfies
\[
   a_t=(1+x+x^2)a_x,
   \qquad
   a(x,0)=x.
\]
Writing \(\theta=\sqrt3t/2\), one obtains
\[
   a(x,t)
   =\frac{
      \sqrt3x\cos\theta+(x+2)\sin\theta
   }{
      \sqrt3\cos\theta-(2x+1)\sin\theta
   }.
\]
Since the bridge is closed,
\[
   xw(x,t)=a(x,t)-a(0,t),
\]
and therefore
\[
   w(x,t)
   =\frac{3}{
      (\sqrt3\cos\theta-\sin\theta)
      (\sqrt3\cos\theta-(2x+1)\sin\theta)
   },
\]
\[
   W(t)
   =\frac{3}{(\sqrt3\cos\theta-\sin\theta)^2}.
\]
Thus the complete original triangle is recovered by solving the local
auxiliary equation and deleting its virtual coefficient once, at the end.
\end{example}

\begin{figure}[ht]
\centering
\includegraphics[width=.47\textwidth]{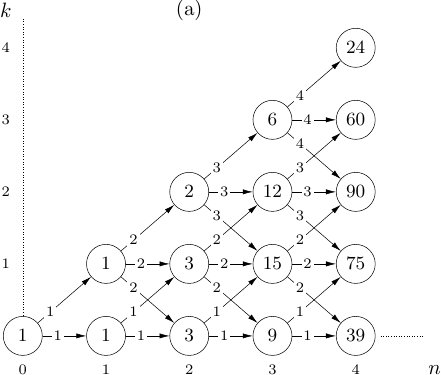}\hfill
\includegraphics[width=.47\textwidth]{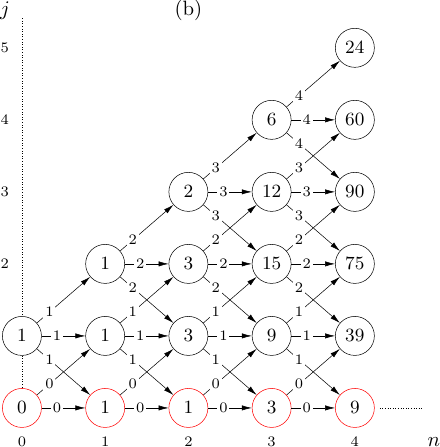}
\caption{Closed bridge at a one-level shift.  Panel (a) is the original
triangle; panel (b) is the free shifted triangle.  Its visible entries are
exactly those of panel (a), shifted upward by one level.}
\label{fig:cascade-basement}
\end{figure}

\begin{example}[An open bridge]
Take
\[
   A=B=C=1,
   \qquad
   \alpha_0=\beta_0=2,
   \qquad
   \gamma_0=1.
\]
Again \(m=1\), but now \(\alpha_0-A=1\), so the bridge \(0\to1\) is open.
The original rows begin
\[
   P_0(x)=1,
   \qquad
   P_1(x)=1+2x,
   \qquad
   P_2(x)=5+6x+6x^2,
\]
whereas free shifted evolution gives
\[
   a_1(x)=1+x+2x^2,
   \qquad
   a_2(x)=1+6x+6x^2+6x^3.
\]
The extra unit in the coefficient of \(x\) comes from the excursion
\[
   1\longrightarrow0\longrightarrow1.
\]
Figure~\ref{fig:feedback-strip}(a) shows the unrestricted auxiliary
triangle.  Panel~(b) shows the correct shifted original triangle: at each
step the dashed, crossed-out node is the newly generated virtual
coefficient, which is discarded before the next row is formed.

For this example \(F_1(t)=a(0,t)\) and \(L_1(t)=b(0,t)\).  The local
recurrences give \(f_3=6\), \(\lambda_1=0\), and \(\lambda_2=1\), so
\eqref{eq:integer-return-coefficient-recursion} yields
\[
   W_2=f_3-\lambda_1W_1-\lambda_2W_0=6-0-1=5.
\]
Continuing the same recursion gives
\[
   W_0,W_1,W_2,W_3,W_4,\ldots
   =1,1,5,17,97,\ldots.
\]
The correction in the return equation removes exactly the same feedback
which produces the spurious coefficient \(6\) in the free auxiliary
triangle.
\end{example}

\begin{figure}[ht]
\centering
\includegraphics[width=.47\textwidth]{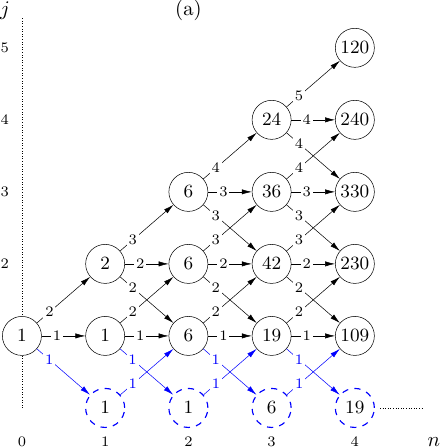}\hfill
\includegraphics[width=.47\textwidth]{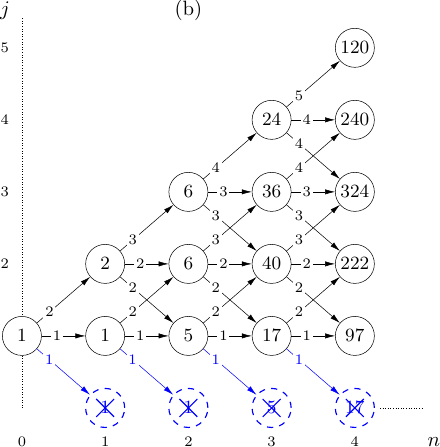}
\caption{Open bridge at a one-level shift.  Panel (a) shows unrestricted
auxiliary evolution.  In panel (b), each dashed crossed-out node is the
virtual contribution generated from the corrected visible row and deleted
before the next step.}
\label{fig:feedback-strip}
\end{figure}

\begin{remark}[Integer index versus integer weights]
A finite shifted floor is controlled by
\[
   \frac{\beta_0-C}{C}=m\in\mathbb Z_{>0},
\]
not by integrality of all step weights.  For example,
\(\beta_k=\tfrac12k+\tfrac32\) has boundary index \(m=2\), although its
weights are not all integers, whereas \(\beta_k=3k+5\) has integer weights
but boundary index \(2/3\), and hence no finite shifted floor.
\end{remark}

The next section shows that the bridge condition \(\alpha_0=A\) has a
second, independent consequence: without any integrality assumption on
\(\nu\), it is exactly the condition for factorisation by terminal height.


\section{Factorisation by terminal height}
\label{sec:collapse}

Example~\ref{ex:ck-nu2-boundary-equation} gave
\[
   w(x,t)=\frac{\sec^3 t}{1-x\tan t}.
\]
Expanding in the height variable shows that its terminal-height columns are
\[
   H_k(t)=\sec^3 t\,\tan^k t.
\]
Thus all columns are obtained from the return series by successive
multiplication by the same function.  We now ask when such a factorisation
can occur for a general height-dependent weighted Motzkin triangle.

For \(k\ge0\), define
\[
   H_k(t)
   :=
   \sum_{n\ge0}w_{n,k}\frac{t^n}{n!}.
\]
Then
\[
   H_0(t)=W(t),
   \qquad
   w(x,t)=\sum_{k\ge0}H_k(t)x^k.
\]
Summing the last-step recurrence \eqref{eq:motzkin-rec} over the length gives
\begin{equation}
\label{eq:collapse-column-recurrence}
   H_k'
   =
   \alpha_{k-1}H_{k-1}
   +\gamma_kH_k
   +\beta_kH_{k+1},
   \qquad k\ge0,
\end{equation}
where the first term is absent for \(k=0\).  The initial conditions are
\[
   H_0(0)=1,
   \qquad
   H_k(0)=0\quad(k\ge1).
\]

Suppose first that the step weights are arbitrary functions of height and
that the columns factor as
\[
   H_k(t)=W(t)Y(t)^k,
   \qquad k\ge0,
\]
for some formal power series \(Y\).  Necessarily \(W=H_0\), and the initial
conditions give
\[
   W(0)=1,
   \qquad
   Y(0)=0.
\]
We determine what this assumption forces before returning to the affine
family.

At height \(k=0\), equation \eqref{eq:collapse-column-recurrence} gives
\begin{equation}
\label{eq:collapse-W-equation}
   W'=(\gamma_0+\beta_0Y)W,
   \qquad
   W(0)=1.
\end{equation}
For \(k\ge1\), cancelling the nonzero factor \(WY^{k-1}\) from the column
equation yields
\begin{equation}
\label{eq:collapse-general-k-equation}
   kY'
   =\alpha_{k-1}
    +(\gamma_k-\gamma_0)Y
    +(\beta_k-\beta_0)Y^2.
\end{equation}
In particular, the equation at height \(1\) is
\begin{equation}
\label{eq:collapse-general-Y-equation}
   Y'
   =\alpha_0
    +(\gamma_1-\gamma_0)Y
    +(\beta_1-\beta_0)Y^2.
\end{equation}
Subtracting \(k\) times \eqref{eq:collapse-general-Y-equation} from
\eqref{eq:collapse-general-k-equation}, we obtain
\[
\begin{aligned}
0={}&\alpha_{k-1}-k\alpha_0\\
&+\bigl(\gamma_k-\gamma_0-k(\gamma_1-\gamma_0)\bigr)Y\\
&+\bigl(\beta_k-\beta_0-k(\beta_1-\beta_0)\bigr)Y^2.
\end{aligned}
\]
If \(H_1=WY\not\equiv0\), then \(Y\not\equiv0\) and \(Y(0)=0\).  The
substitution map from polynomials in an indeterminate \(z\) to formal series,
\(z\mapsto Y(t)\), is therefore injective: the lowest power of a nonzero
polynomial remains visible after substitution.  The three coefficients in
the last display must consequently vanish.

\begin{theor_eng}[Rigidity of terminal-height factorisation]
\label{thm:terminal-height-factorisation}
Consider an arbitrary height-dependent weighted Motzkin rule, and assume
\(H_1\not\equiv0\).  The following statements are equivalent:
\begin{enumerate}
\item there is a formal power series \(Y(0)=0\) such that
\[
   H_k(t)=H_0(t)Y(t)^k
   \qquad(k\ge0);
\]
\item the step weights necessarily have the form
\begin{equation}
\label{eq:factorisation-rigidity-weights}
   \alpha_k=(k+1)\alpha_0,
   \qquad
   \gamma_k=\gamma_0+k(\gamma_1-\gamma_0),
   \qquad
   \beta_k=\beta_0+k(\beta_1-\beta_0).
\end{equation}
\end{enumerate}
When these conditions hold, \(W=H_0\) and \(Y\) are uniquely determined by
\[
   W'=(\gamma_0+\beta_0Y)W,
   \qquad W(0)=1,
\]
\[
   Y'=\alpha_0+(\gamma_1-\gamma_0)Y+(\beta_1-\beta_0)Y^2,
   \qquad Y(0)=0,
\]
and
\begin{equation}
\label{eq:collapse-column-factorisation}
   H_k(t)=W(t)Y(t)^k,
   \qquad k\ge0.
\end{equation}
Consequently,
\begin{equation}
\label{eq:collapse-w}
   w(x,t)=\frac{W(t)}{1-xY(t)}.
\end{equation}
\end{theor_eng}

\begin{proof}
The necessity was established above.  Conversely, assume
\eqref{eq:factorisation-rigidity-weights}, let \(W,Y\) solve the displayed
one-variable equations, and put \(\widehat H_k=WY^k\).  For \(k=0\), the
required column equation is the equation for \(W\).  For \(k\ge1\),
\begin{align*}
   \widehat H_k'
   &=W'Y^k+kWY^{k-1}Y'\\
   &=k\alpha_0\,WY^{k-1}
     +\bigl(\gamma_0+k(\gamma_1-\gamma_0)\bigr)WY^k\\
   &\qquad
     +\bigl(\beta_0+k(\beta_1-\beta_0)\bigr)WY^{k+1},
\end{align*}
which is precisely the column recurrence.  The initial conditions also
agree with the path triangle, so coefficientwise uniqueness gives
\(\widehat H_k=H_k\) for every \(k\).  Summing in \(x\) gives
\eqref{eq:collapse-w}.
\end{proof}

For the affine weights
\(\alpha_k=Ak+\alpha_0\), \(\beta_k=Ck+\beta_0\), and
\(\gamma_k=Bk+\gamma_0\), the conditions on \(\beta_k\) and
\(\gamma_k\) in
\eqref{eq:factorisation-rigidity-weights} already hold.  The remaining
condition is
\[
   Ak+\alpha_0=(k+1)\alpha_0,
\]
equivalently \(\alpha_0=A\).  In fact the equation at height \(2\) already
detects this condition, since it gives
\[
   (\alpha_0-A)WY=0.
\]
Thus, in the nondegenerate affine model,
\begin{equation}
\label{eq:collapse-Y-equation}
   Y'=A+BY+CY^2,
   \qquad
   Y(0)=0,
\end{equation}
and factorisation holds if and only if \(\alpha_0=A\).

The rigidity theorem does not assume affine weights in advance.  Within
the affine family it reduces to \(\alpha_0=A\), and the first three
terminal-height columns already detect this condition.  At integer boundary
index, Section~\ref{sec:integer-shift} found the same equality by a
different argument: it closes the bridge from the virtual strip back to the
visible region.  Thus the bridge criterion and the affine specialisation of
the factorisation criterion arise independently and coincide.

Because \(H_0(0)=1\), the common ratio is determined by the first two
columns:
\[
   Y(t)=\frac{H_1(t)}{H_0(t)},
   \qquad
   H_k(t)=H_0(t)\left(\frac{H_1(t)}{H_0(t)}\right)^k.
\]
Thus the entire height-refined triangle is determined by its return column
and its height-one column.  After dividing the \(k\)-th column by
\(k!\), one obtains the exponential Riordan array \([W,Y]\).  The
factorised Riordan form is classical \cite{Barry2011,BarryMwafise2018}.
The content of Theorem~\ref{thm:terminal-height-factorisation} is the direct
path-column rigidity statement: geometric factorisation forces the complete
arrival-proportional affine rule.  For the affine model of this paper, the
resulting condition is exactly the virtual-bridge criterion.

The two one-variable equations can be solved explicitly.  The following
form is convenient when \(C\ne0\).

\begin{corollary}[Explicit factorised form]
\label{cor:collapse-explicit-form}
Assume \(\alpha_0=A\) and \(C\ne0\), and put
\[
   p:=\frac{\beta_0}{C}.
\]
Choose either square root \(\kappa\) of
\[
   \kappa^2=B^2-4AC,
\]
and define
\begin{equation}
\label{eq:collapse-D-def}
   D(t)
   :=
   \cosh\frac{\kappa t}{2}
   -\frac{B}{\kappa}\sinh\frac{\kappa t}{2},
\end{equation}
with the Taylor-limit interpretation when \(\kappa=0\).  Then
\begin{equation}
\label{eq:collapse-explicit-Y}
   Y(t)
   =
   \frac{2A}{\kappa}
   \frac{\sinh(\kappa t/2)}{D(t)},
\end{equation}
\begin{equation}
\label{eq:collapse-explicit-W}
   W(t)
   =
   e^{(\gamma_0-pB/2)t}D(t)^{-p},
\end{equation}
and \(H_k=WY^k\).  The resulting Taylor series are independent of the
choice of sign of \(\kappa\).  When \(\kappa=0\),
\begin{equation}
\label{eq:collapse-double-root}
   D(t)=1-\frac{B}{2}t,
   \qquad
   Y(t)=\frac{At}{1-Bt/2}.
\end{equation}
\end{corollary}

The calculation is given in Appendix~\ref{app:riccati-collapse}.  Since
\(D(0)=1\), the power \(D^{-p}\) is unambiguous as a formal Taylor series.
In the positive-index regime, \(p=\nu+1\).

The continuous boundary parameter also has a direct coefficientwise
meaning which does not require factorisation.

\begin{proposition}[Polynomial dependence on the boundary parameter]
Fix the up- and level-step weights and a nonzero constant \(C\), and write
\[
   \beta_h=C(h+p).
\]
For every \(n,k\ge0\), the path total \(w_{n,k}(p)\) is a polynomial in
\(p\), with
\[
   \deg_p w_{n,k}(p)
   \le \left\lfloor\frac{n-k}{2}\right\rfloor.
\]
If the remaining step weights and \(C\) are nonnegative, then this polynomial
has nonnegative coefficients.
\end{proposition}

\begin{proof}
A path ending at height \(k\) and having \(d\) down-steps satisfies
\(d\le\lfloor(n-k)/2\rfloor\).  Each down-step contributes one factor
\(C(h+p)\), linear in \(p\), so the weight of the path has degree at most
\(d\).  Expanding each factor as \(Ch+Cp\) also gives the coefficientwise
nonnegativity and the interpretation by marked down-steps.  For \(k=0\) and
even \(n=2d\) the bound is attained: the alternating path \((UD)^d\)
contributes \(\prod_{j=0}^{d-1}\alpha_j\cdot C^d p^d\), and no cancellation
is possible when the up-step weights are positive.
\end{proof}

The Dyck example from Section~\ref{sec:continuous-kernel} is the case
\(p=3\) of a full one-parameter family.  Take
\begin{equation}
\label{eq:collapse-dyck-weights}
   \alpha_k=k+1,
   \qquad
   \beta_k=k+p,
   \qquad
   \gamma_k=0.
\end{equation}
Here \(\alpha_0=A=1\), so the factorisation criterion applies.  Equations
\eqref{eq:collapse-Y-equation}--\eqref{eq:collapse-W-equation} reduce to
\[
   Y'=1+Y^2,
   \qquad
   \frac{W'}{W}=pY,
   \qquad
   Y(0)=0,
   \quad
   W(0)=1.
\]
Hence
\[
   Y(t)=\tan t,
   \qquad
   W(t)=\sec^p t.
\]

\begin{corollary}[Dyck paths and powers of the secant]
For the weights \eqref{eq:collapse-dyck-weights},
\[
   H_k(t)=\sec^p t\,\tan^k t,
   \qquad
   w(x,t)=\frac{\sec^p t}{1-x\tan t}.
\]
Equivalently,
\[
   w_{n,k}
   =
   n!\,[t^n] \,\sec^p t\,\tan^k t.
\]
\end{corollary}

The return-column specialisation \(W(t)=\sec^p t\), together with its
higher-order Euler and orthogonal-polynomial interpretations, is classical;
see \cite{Barry2012,JiuShi2019}.  For every real \(p\ge0\), these are
nonnegative weighted Dyck paths; for integer \(p\), the step weights may be
read as numbers of colours.  The case \(p=3\) recovers
Example~\ref{ex:ck-nu2-boundary-equation}.

For the secant family, the degree bound is sharp for returns of length
\(2n\): the alternating path \((UD)^n\) contributes the leading term
\(p^n\).

The same criterion gives genuinely Motzkin examples with level steps.  Take
\[
   \alpha_k=k+1,
   \qquad
   \beta_k=k+p,
   \qquad
   \gamma_k=k+1.
\]
Put
\[
   D(t)
   :=
   \cos\frac{\sqrt3t}{2}
   -\frac1{\sqrt3}\sin\frac{\sqrt3t}{2}.
\]
Corollary~\ref{cor:collapse-explicit-form} gives
\begin{equation}
\label{eq:collapse-motzkin-YW}
   Y(t)
   =
   \frac{2}{\sqrt3}
   \frac{\sin(\sqrt3t/2)}{D(t)},
   \qquad
   W(t)
   =
   e^{(1-p/2)t}D(t)^{-p}.
\end{equation}
For \(p=2\), the exponential factor disappears and
\[
   W(t)=D(t)^{-2}
   =1+t+3\frac{t^2}{2!}+9\frac{t^3}{3!}
    +39\frac{t^4}{4!}+189\frac{t^5}{5!}+\cdots.
\]
These are the return weights of the one-level closed-bridge example in
Section~\ref{sec:integer-shift}; formula
\eqref{eq:collapse-motzkin-YW} simultaneously gives every positive terminal
height.

The elementary degeneration \(C=0\) is recorded in
Appendix~\ref{app:riccati-collapse}.

Thus \(\alpha_0=A\) has two manifestations.  At integer boundary index it
makes the unique bridge out of the virtual strip vanish; for arbitrary
boundary parameter it is exactly the nondegenerate condition for
factorisation of every terminal-height column.


\section{Jacobi data and a boundary-perturbed birth--death model}
\label{sec:external-readings}

Sections~\ref{sec:continuous-kernel}--\ref{sec:collapse} use the oriented
step weights to determine both the return series and the full
height-refined triangle.  We close the main development with two
complementary viewpoints.  First, we determine exactly how much of the
oriented weighting is retained when only return paths are recorded.  This
leads to the standard Jacobi continued fraction and an explicit affine
Hankel product.  Second, we consider a linear birth--death model with a
single perturbation at the floor.  Its boundary equation has the same
characteristic mechanism, and its boundary index is a ratio of transition
rates, which is naturally noninteger.

\subsection{Return invariance and Jacobi data}

The weights attached to the two directions of the same height edge may be
redistributed without changing any return weight.

\begin{proposition}[Rescaling the two directions of an edge]
Let \(q_h\ne0\) for \(h\ge0\), and replace the step weights by
\[
   \widetilde\alpha_h=q_h\alpha_h,
   \qquad
   \widetilde\beta_h=q_h^{-1}\beta_h,
   \qquad
   \widetilde\gamma_h=\gamma_h.
\]
If \(\widetilde w_{n,k}\) denotes the resulting path total, then
\begin{equation}
\label{eq:edge-rescaling-endpoint}
   \widetilde w_{n,k}
   =\left(\prod_{h=0}^{k-1}q_h\right)w_{n,k},
\end{equation}
with the empty product equal to \(1\).  In particular,
\[
   \widetilde w_{n,0}=w_{n,0}.
\]
\end{proposition}

\begin{proof}
For a path ending at height \(k\), let \(U_h\) and \(D_h\) be the numbers of
up- and down-crossings of the edge between heights \(h\) and \(h+1\).
Then
\[
   U_h-D_h=
   \begin{cases}
      1,&0\le h<k,\\
      0,&h\ge k.
   \end{cases}
\]
The ratio of the new path weight to the old one is
\(\prod_hq_h^{U_h-D_h}\), which is the factor in
\eqref{eq:edge-rescaling-endpoint}.
\end{proof}

For a return path, let \(L_h\) be the number of level-steps at height
\(h\).  Since the numbers of up- and down-crossings of the edge
\(\{h,h+1\}\) are equal, say \(U_h=D_h\), the path weight is
\[
   \prod_{h\ge0}\gamma_h^{L_h}
   \prod_{h\ge0}(\alpha_h\beta_h)^{U_h}.
\]
Thus the return sequence depends on the oriented edge weights only through
the Jacobi data
\[
   \mathfrak b_h:=\gamma_h,
   \qquad
   \Lambda_{h+1}:=\alpha_h\beta_h,
   \qquad h\ge0.
\]
This direct argument also covers vanishing edge weights.  Return paths retain
the level weights and the products of opposite edge weights, but not their
separate orientation.  Equivalently, reversing a return path turns an
up-step \(h\to h+1\), of weight \(\alpha_h\), into a down-step
\(h+1\to h\), whose arrival-height weight is \(\beta_h\), and conversely.
The positive-height columns do retain the orientation, as the factor in
\eqref{eq:edge-rescaling-endpoint} shows.  In particular, the factorisation
criterion of Section~\ref{sec:collapse} is strictly stronger than a
statement about the return sequence alone.

For affine weights,
\begin{equation}
\label{eq:affine-jacobi-products}
   \Lambda_{k+1}
   =(Ak+\alpha_0)(Ck+\beta_0).
\end{equation}
The quadratic polynomial on the right has a factor \(k+1\) precisely when
\[
   (\alpha_0-A)(\beta_0-C)=0.
\]
Hence, in the positive-index regime \(\beta_0\ne C\), the bridge condition
\(\alpha_0=A\) is also exactly the condition under which path reversal
produces a balanced affine model with the same return sequence.  This
observation concerns \(W\) only; the height-refined triangle changes under
reversal.

\begin{remark}[Dual boundary index]
Assume \(A\ne0\) and define
\[
   \nu^{\vee}:=\frac{\alpha_0-A}{A}
              =\frac{\alpha_0}{A}-1.
\]
Path reversal exchanges \((A,\alpha_0)\) with \((C,\beta_0)\), so
\(\nu^{\vee}\) is the boundary index of the reversed affine model.  Hence,
when \(A>0\) and \(\nu^{\vee}\in\mathbb Z_{>0}\), the same return sequence
has a finite shifted-floor description after reversal, even if the original
index \(\nu\) is noninteger.  In the hypergeometric parameters of
Appendix~\ref{app:continuous-kernel-details},
\(a=\nu^{\vee}+1\) and \(b=\nu+1\); the symmetry \(a\leftrightarrow b\)
reflects this reversal.
\end{remark}

Ordinary generating functions are natural for first-return decomposition,
so define the formal series
\[
   G(z):=\sum_{n\ge0}W_nz^n.
\]
For each height \(h\), let \(G_h(z)\) enumerate excursions which start and
end at \(h\) and never go below \(h\), with the inherited weights.  A
nonempty excursion begins either with a level-step at \(h\), or with an
up-step, an excursion based at \(h+1\), and the matching down-step; in both
cases an arbitrary excursion based at \(h\) may follow.  Therefore
\[
   G_h
   =1+\mathfrak b_hzG_h
    +\Lambda_{h+1}z^2G_{h+1}G_h.
\]
Iterating from \(G_0=G\) gives the Jacobi continued fraction
\[
   G(z)
   =
   \cfrac{1}{1-\mathfrak b_0z-
      \cfrac{\Lambda_1z^2}{1-\mathfrak b_1z-
      \cfrac{\Lambda_2z^2}{1-\mathfrak b_2z-\ddots}}}.
\]
This is the Flajolet--Viennot correspondence in the present notation
\cite{Flajolet1980CF,Viennot1985}.  When the data are numerical with
\(\mathfrak b_h\in\mathbb R\) and \(\Lambda_h>0\), they are also the
recurrence coefficients of a Jacobi matrix and \((W_n)\) is its moment
sequence \cite{Chihara1978,Ismail2005}.

\begin{proposition}[Return-equivalent oriented models]
Let
\[
   \Delta_N:=\det(W_{i+j})_{i,j=0}^{N},
\]
and assume \(\Delta_N\ne0\) for every \(N\ge0\).  Then the return sequence
uniquely determines the Jacobi data
\[
   \gamma_h,
   \qquad
   \alpha_h\beta_h
   \qquad(h\ge0).
\]
Consequently, two oriented weighted Motzkin models with this quasi-definite
return sequence have the same returns if and only if these data agree.  If
all oriented edge weights are nonzero, the two models differ exactly by a
reciprocal edge rescaling
\[
   \widetilde\alpha_h=q_h\alpha_h,
   \qquad
   \widetilde\beta_h=q_h^{-1}\beta_h.
\]
\end{proposition}

\begin{proof}
The direct path argument above proves the ``if'' direction.  Under the
nonvanishing Hankel hypothesis, the moment functional is quasi-definite and
its Jacobi coefficients are unique \cite{Chihara1978,Ismail2005}; hence the
return sequence determines \(\mathfrak b_h=\gamma_h\) and
\(\Lambda_{h+1}=\alpha_h\beta_h\).  If all edge weights are nonzero, take
\(q_h=\widetilde\alpha_h/\alpha_h\); equality of the edge products then
gives the stated formula for \(\widetilde\beta_h\).
\end{proof}

For affine weights, a quasi-definite return sequence therefore determines
the five combinations
\[
   B,\qquad \gamma_0,\qquad AC,\qquad
   A\beta_0+C\alpha_0,\qquad \alpha_0\beta_0,
\]
which are exactly the coefficients of \(\gamma_k\) and of the quadratic
polynomial \((Ak+\alpha_0)(Ck+\beta_0)\).

The classical Hankel determinant formula for Jacobi moments
\cite{Flajolet1980CF,Chihara1978,Ismail2005,Krattenthaler2015}, specialised
to \eqref{eq:affine-jacobi-products}, gives the following affine product.

\begin{corollary}[Hankel product for affine returns]
For \(N\ge0\),
\begin{equation}
\label{eq:general-affine-hankel}
   \det\bigl(W_{i+j}\bigr)_{i,j=0}^{N}
   =
   \prod_{r=1}^{N}
   \left[
      (A(r-1)+\alpha_0)(C(r-1)+\beta_0)
   \right]^{N+1-r}.
\end{equation}
\end{corollary}

\begin{proof}
For a Jacobi moment sequence with edge products \(\Lambda_r\), the
classical formula is
\[
   \det\bigl(W_{i+j}\bigr)_{i,j=0}^{N}
   =\prod_{r=1}^{N}\Lambda_r^{\,N+1-r}.
\]
Substitution of \eqref{eq:affine-jacobi-products} gives
\eqref{eq:general-affine-hankel}.
\end{proof}

A zero edge product therefore forces the corresponding Hankel determinants
to vanish.  The same Jacobi recurrence gives a finite polynomial expression
for the reciprocal ordinary series at every fixed terminal height.  For
generic affine parameters, the associated tails telescope further to one
shifted \({}_2F_1\) numerator at each height; both formulas are recorded in
Appendix~\ref{app:continued-fraction-details}.

At integer boundary index, the shifted recurrence of
Section~\ref{sec:integer-shift} has coupling across the top of the virtual
strip
\[
   \widetilde\Lambda_m
   =\widetilde\beta_{m-1}\widetilde\alpha_{m-1}
   =mC(\alpha_0-A).
\]
Thus the closed-bridge condition is exactly the point at which the
associated continued-fraction tail decouples from the visible recurrence.
The finite transfer formula through the virtual strip is given in
Appendix~\ref{app:continued-fraction-details}.

\subsection{A boundary-perturbed birth--death model}

Let \(p_k(t)\) be the state probabilities of a continuous-time birth--death
chain on \(\mathbb Z_{\ge0}\), with birth rates \(\lambda_k\), death rates
\(\mu_k\), and \(\mu_0=0\).  Assume first that the chain starts from state
\(0\), so that \(p_k(0)=\delta_{k0}\).  Its forward equations are
\begin{equation}
\label{eq:bd-forward}
   p'_k(t)
   =\lambda_{k-1}p_{k-1}(t)+\mu_{k+1}p_{k+1}(t)
    -(\lambda_k+\mu_k)p_k(t),
   \qquad k\ge0,
\end{equation}
where missing boundary terms are omitted.  If
\(p_k(t)=\sum_{n\ge0}c_{n,k}t^n/n!\), then the coefficients \(c_{n,k}\)
satisfy the Motzkin recurrence with signed weights
\begin{equation}
\label{eq:bd-path-weights}
   \alpha_k=\lambda_k,
   \qquad
   \beta_k=\mu_{k+1},
   \qquad
   \gamma_k=-(\lambda_k+\mu_k).
\end{equation}
With the stated initial condition, the coefficients \(c_{n,k}\) are
exactly the signed weighted-path totals associated with
\eqref{eq:bd-path-weights}.  For a general initial distribution they are the
corresponding linear combination of triangles starting at the possible
initial states.  The negative level weight is the diagonal term of the
Markov generator, not an additional positive colour.  This is the
weighted-path form of the Karlin--McGregor correspondence
\cite{KarlinMcGregor1957,ChiharaIsmail1982,FlajoletGuillemin2000}.

Consider now
\[
   \lambda_k=\lambda,
   \qquad
   \mu_k=\mu+\theta k\quad(k\ge1),
   \qquad
   \mu_0=0,
\]
with \(\lambda,\mu,\theta>0\), and suppose that the chain starts empty.
The interior coefficients are affine, but the reflecting floor creates one
exception:
\[
   \gamma_0=-\lambda,
   \qquad
   \gamma_k=-(\lambda+\mu+\theta k)\quad(k\ge1).
\]
Thus this queue is a one-boundary perturbation of an affine weighted-path
model, rather than a literal member of the six-parameter family at every
height.

Let
\[
   \mathcal G(x,t):=\sum_{k\ge0}p_k(t)x^k.
\]
Summing \eqref{eq:bd-forward} gives
\begin{equation}
\label{eq:bd-pgf}
   \mathcal G_t
   =\theta(1-x)\mathcal G_x
    +\lambda(x-1)\mathcal G
    +\mu(1-x)
      \frac{\mathcal G(x,t)-\mathcal G(0,t)}{x}.
\end{equation}
Equivalently,
\begin{align*}
   \mathcal G_t
   &={}
   \theta(1-x)\mathcal G_x
   +(\lambda x-\lambda-\mu)\mathcal G
   +\mu\frac{\mathcal G(x,t)-\mathcal G(0,t)}{x}
   +\mu\mathcal G(0,t).
\end{align*}
The first three terms are an affine positive-index equation with
\[
   A=0,
   \quad B=-\theta,
   \quad C=\theta,
   \quad \alpha_0=\lambda,
   \quad \gamma_0=-(\lambda+\mu),
   \quad \delta=\mu;
\]
the final term is the single correction at the floor.  Moreover,
\[
   \beta_k=\mu_{k+1}
   =\theta\left(k+1+\frac{\mu}{\theta}\right),
\]
so the boundary index of the affine interior rule is
\[
   \nu=\frac{\mu}{\theta}.
\]
This gives a direct probabilistic reason for treating \(\nu\) as a
continuous parameter: it is a ratio of transition rates and need not be an
integer.

The characteristic-and-Volterra analysis of this queue goes back to
Zheng, Chao, and Ji \cite{ZhengChaoJi2004}; that analytic method is not
claimed as new here.  Its role in the present paper is to exhibit the
boundary index as a natural rate ratio in a one-floor perturbation of the
affine path equation.  Applying the same boundary-flow calculation as in
Section~\ref{sec:continuous-kernel} recovers their Kummer quotient for the
empty-state probability.  If
\[
   a=\frac{\lambda}{\theta},
   \qquad
   \nu=\frac{\mu}{\theta},
   \qquad
   \widehat p_0(s)=\int_0^\infty e^{-st}p_0(t)\,dt,
\]
then, for \(\Re s>0\),
\begin{equation}
\label{eq:kummer-transform}
   \widehat p_0(s)
   =
   \frac{1}{s}
   \frac{
      {}_1F_1\!\bigl(\nu+1;\nu+1+s/\theta;-a\bigr)
   }{
      {}_1F_1\!\bigl(\nu;\nu+1+s/\theta;-a\bigr)
   }.
\end{equation}
A derivation, including the effect of the floor perturbation, is given in
Appendix~\ref{app:kummer-bd}.

These two viewpoints clarify different aspects of the preceding theory.
The Jacobi data show exactly what survives when the full oriented triangle
is reduced to its return column: the level weights and the products of
opposite edge weights remain, while their orientation is lost.  The birth--death model shows that the boundary index arises naturally as
the continuous rate ratio \(\mu/\theta\), while a literal finite shifted
floor is available only at integer values.


\section{Conclusion}
\label{sec:conclusion}

For affine weighted Motzkin paths with positive boundary index, the
height-refined generating-function equation contains one unknown scalar
series, namely the return series \(W(t)\).  The characteristic identity of
Section~\ref{sec:continuous-kernel} resolves both parts of this boundary
problem.  At the boundary it gives a one-dimensional Abel--Volterra
equation which determines \(W\) from the affine step weights; at an interior
point the same identity reconstructs the full generating function
\(w(x,t)\).  Thus the return column is neither external data nor a separate
enumerative input: it is selected by the affine path rule and, once
selected, determines every terminal-height column.  The term
\emph{differential kernel method} refers precisely to the boundary
cancellation which produces this return equation.

When the boundary index is a positive integer \(m\), the same mechanism has
a finite combinatorial realisation.  Shifting every height by \(m\)
replaces the divided-difference evolution by a local auxiliary evolution on
a half-line with \(m\) virtual levels.  The first-entry decomposition of
Section~\ref{sec:integer-shift} expresses the original triangle through two
local shifted path triangles; at the boundary it recovers the return
equation, and after \(m\) differentiations it gives a Volterra equation of
the second kind.  The only bridge from the virtual strip back to the
visible region has weight \(\alpha_0-A\).  Consequently, intermediate
deletions can be postponed until the final row if and only if
\(\alpha_0=A\).

Terminal-height factorisation is rigid even without an affine hypothesis:
if \(H_k=WY^k\) for every height and \(H_1\not\equiv0\), then the full
step rule is forced to be arrival-proportional affine.  Within the affine
family this classification reduces to \(\alpha_0=A\).  The same algebraic
condition has two combinatorial consequences: for arbitrary boundary
parameter it gives terminal-height factorisation, while at positive integer
boundary index it additionally closes the return bridge from the virtual
strip.  The secant-power family is a transparent instance of this general
structure.

The final section clarifies what is retained under two standard external
readings.  The return sequence depends on the oriented edge weights only
through the Jacobi products \(\alpha_h\beta_h\), together with the level
weights \(\gamma_h\); under the quasi-definite Hankel condition these data
are also uniquely recoverable from the returns.  The full height-refined
triangle retains the missing orientation data.  The boundary-perturbed birth--death model gives a
different interpretation of the boundary index: it arises naturally as a
continuous ratio of transition rates, while a literal finite shifted floor
is available only at integer values.

Several boundary regimes have deliberately been left outside the main line.
The balanced case \(\beta_0=C\) is local and belongs to the separate local
theory.  Negative index, \(\beta_0<C\), is not covered by the positive Abel
kernel used here and would require a boundary analysis of Frobenius type.
The normalisation by the boundary index does not apply when \(C=0\); for the
arrival-proportional subfamily \(\alpha_0=A\),
Section~\ref{sec:collapse} still gives closed forms, but outside it the case
\(C=0\) is a separate regime.  In the constant-weight Motzkin case
\(A=B=C=0\), for instance, the return series is
\[
   W(t)=e^{\gamma_0t}
   \frac{I_1\!\bigl(2\sqrt{\alpha_0\beta_0}\,t\bigr)}
        {\sqrt{\alpha_0\beta_0}\,t},
\]
with the usual limiting interpretation when \(\alpha_0\beta_0=0\).  More
general perturbations at the floor are a natural further direction.

%
\section*{Acknowledgements}

I am deeply grateful to Cyril Banderier for many hours of detailed and
generous discussion of this material, and above all for his advice on how
to present it to a lattice-path audience.  I also thank the participants of the Lattice Paths
Combinatorics conference (Vienna, July 2026) for their questions and
comments.


\appendix


\section{Analytic and special-function details for the return equation}
\label{app:continuous-kernel-details}

This appendix supplies the analytic details used in
Section~\ref{sec:continuous-kernel} and records two optional consequences of
the one-dimensional return equation: an ordinary power-series quotient and,
for generic affine parameters, a hypergeometric evaluation.

\begin{lemma}[Analyticity of the exponential generating function]
\label{lem:app-egf-analytic}
Let the six affine parameters be complex numbers and put
\[
   M=\max(1,|A|,|B|,|C|,|\alpha_0|,|\beta_0|,|\gamma_0|).
\]
Then
\[
   |w_{n,k}|
   \le 3^nM^n(n+1)^n
   \le e\,n!\,(3Me)^n.
\]
Consequently, for every \(R>0\), the series
\eqref{eq:motzkin-egf} converges absolutely for
\[
   |t|<\frac{1}{3Me(1+R)},
   \qquad
   |x|\le R.
\]
In particular, \(w(x,t)\) is analytic near \(t=0\), locally uniformly in
\(x\), and \(W(t)=w(0,t)\) is analytic near the origin.
\end{lemma}

\begin{proof}
A path of length \(n\) never rises above height \(n\).  Hence every step
used by such a path has absolute weight at most \(M(n+1)\), and there are at
most \(3^n\) step sequences.  This gives the first estimate.  The inequality
\((n+1)^n\le e^{n+1}n!\) gives the second.

For \(|x|\le R\),
\[
   |P_n(x)|
   \le
   (n+1)\max_k|w_{n,k}|\max(1,R)^n
   \le
   e(n+1)n!\bigl(3Me(1+R)\bigr)^n.
\]
The asserted convergence and termwise differentiation follow.
\end{proof}

Lemma~\ref{lem:app-egf-analytic}, together with
\eqref{eq:ck-FL-leading}, justifies the boundary limit used in the derivation
of \eqref{eq:ck-boundary-flow-identity}: as \(\tau\downarrow0\),
\[
   F_\nu(\tau)w(X(\tau),T-\tau)\longrightarrow0.
\]
The remaining analytic point is uniqueness for the weakly singular
convolution equation.

\begin{lemma}[Standard Abel regularisation]
\label{lem:app-abel-regularisation}
Let
\[
   L(t)=t^{\nu-1}k(t),
   \qquad
   \nu>0,
   \qquad
   k(0)\ne0,
\]
with \(k\) analytic near \(0\).  Suppose that \(U\) is analytic near
\(0\) and satisfies
\[
   H=L*U,
   \qquad
   (L*U)(t)=\int_0^tL(t-s)U(s)\,ds.
\]
Put
\[
   q=\lceil\nu\rceil,
   \qquad
   \theta=q-\nu\in[0,1),
\]
and let \(I^\theta\) denote the Riemann--Liouville fractional integral,
with \(I^0\) the identity.  If \(J=I^\theta L\), then
\[
   J(t)=t^{q-1}j(t),
   \qquad
   j(0)=\frac{\Gamma(\nu)}{\Gamma(q)}k(0),
\]
and
\begin{equation}
\label{eq:app-regularised-second-kind}
   (I^\theta H)^{(q)}(t)
   =
   \Gamma(\nu)k(0)U(t)
   +\int_0^tJ^{(q)}(t-s)U(s)\,ds.
\end{equation}
Thus every analytic solution of the first-kind equation \(H=L*U\)
satisfies a Volterra equation of the second kind with analytic kernel.  In
particular, the first-kind equation has at most one analytic solution near
\(0\), equivalently at most one formal Taylor-series solution.  If
\(\nu\in\mathbb Z_{>0}\), then \(\theta=0\) and the conversion is ordinary
\(\nu\)-fold differentiation.
\end{lemma}

\begin{proof}
Write \(k(t)=\sum_{r\ge0}k_rt^r\).  Termwise fractional integration gives
\[
   I^{q-\nu}\bigl(t^{\nu-1+r}\bigr)
   =
   \frac{\Gamma(\nu+r)}{\Gamma(q+r)}t^{q-1+r}.
\]
Hence \(J=I^{q-\nu}L=t^{q-1}j(t)\), with
\(j(0)=\Gamma(\nu)k(0)/\Gamma(q)\).  Differentiating \(J*U\) exactly
\(q\) times gives \eqref{eq:app-regularised-second-kind}: all lower
boundary terms vanish, and the only surviving boundary term is
\(\Gamma(\nu)k(0)U(t)\).  Uniqueness for the resulting second-kind
Volterra equation is classical
\cite{GorenfloVessella1991,Brunner2004}.
\end{proof}

For the affine Motzkin kernel,
\[
   L_\nu(t)
   =t^{\nu-1}\bigl(\nu C^\nu+O(t)\bigr).
\]
The path generating function supplies existence of a solution in the
present problem, while Lemma~\ref{lem:app-abel-regularisation} proves the
uniqueness asserted in
Theorem~\ref{thm:boundary-selection-combinatorial}.  At integer index
\(\nu=m\), the same lemma reduces to ordinary \(m\)-fold differentiation,
as used in Section~\ref{sec:integer-shift}.

For completeness, we now record the ordinary-series form of the coefficient
recursion.  With
\[
   F_\nu(t)=t^\nu\sum_{N\ge0}h_Nt^N,
   \qquad
   L_\nu(t)=t^{\nu-1}\sum_{j\ge0}\ell_jt^j,
\]
put
\[
   \widetilde h_N=\Gamma(\nu+N+1)h_N,
   \qquad
   \widetilde\ell_j=\Gamma(\nu+j)\ell_j.
\]
Then Corollary~\ref{cor:ck-beta-recursion} is equivalent to
\[
   \widetilde h_N
   =\sum_{j=0}^{N}\widetilde\ell_jW_{N-j}.
\]
Consequently, for the formal ordinary generating series
\(G(z)=\sum_{n\ge0}W_nz^n\), one has
\begin{equation}
\label{eq:ck-series-quotient}
   G(z)
   =
   \frac{\sum_{N\ge0}\widetilde h_Nz^N}
        {\sum_{j\ge0}\widetilde\ell_jz^j}.
\end{equation}
Thus the fractional power at the origin is absorbed by gamma factors, while
the return enumerator is a formal quotient of power series.

\begin{proposition}[Exact formal ordinary generating series for returns]
Assume in addition that \(A\ne0\) and \(B^2-4AC\ne0\).  Choose
\(\kappa\) with
\[
   \kappa^2=B^2-4AC,
\]
and put
\[
   a=\frac{\alpha_0}{A},
   \qquad
   b=\frac{\beta_0}{C}=\nu+1,
   \qquad
   r=\frac{B+\kappa}{2\kappa},
\]
\[
   c(s)=\frac{s-\gamma_0}{\kappa}+(a+b-1)r.
\]
Define the formal reciprocal Laurent series
\[
   \mathcal R(s)
   :=\frac1sG\left(\frac1s\right)
   =\sum_{n\ge0}\frac{W_n}{s^{n+1}}.
\]
Then, as a formal Laurent series at \(s=\infty\),
\begin{equation}
\label{eq:ck-hypergeometric-return-series}
   \mathcal R(s)
   =
   \frac{1}{\kappa c(s)}
   \frac{
      {}_2F_1\!\left(a,b;c(s)+1;r\right)
   }{
      {}_2F_1\!\left(a-1,b-1;c(s);r\right)
   }.
\end{equation}
The formal branch is the one with \(\mathcal R(s)\sim s^{-1}\).  With this
convention, the Laurent series is independent of the sign chosen for
\(\kappa\).
\end{proposition}

\begin{proof}
For a generalized Taylor series
\[
   f(t)=t^\lambda\sum_{j\ge0}f_jt^j,
   \qquad \lambda>-1,
\]
define its formal Laplace--Borel transform by
\[
   \mathcal L_{\!B}[f](s)
   :=
   \sum_{j\ge0}
      f_j\Gamma(\lambda+j+1)s^{-\lambda-j-1}.
\]
In particular,
\[
   \mathcal L_{\!B}[W](s)
   =\sum_{n\ge0}\frac{W_n}{s^{n+1}}
   =\mathcal R(s).
\]
The beta identity for convolution coefficients gives
\(
   \mathcal L_{\!B}[f*g]
   =\mathcal L_{\!B}[f]\mathcal L_{\!B}[g]
\).
Thus \(F_\nu=L_\nu*W\) implies
\begin{equation}
\label{eq:app-transform-quotient}
   \mathcal R(s)
   =
   \frac{\mathcal L_{\!B}[F_\nu](s)}
        {\mathcal L_{\!B}[L_\nu](s)}.
\end{equation}

Put
\[
   z=1-e^{-\kappa t}.
\]
The solution of \(X'=AX^2+BX+C\), \(X(0)=0\), is
\[
   X(t)=\frac{C}{\kappa}\frac{z}{1-rz}.
\]
Set
\[
   K
   =\frac{\gamma_0}{\kappa}
    +(b-1)(1-r)-ar.
\]
A direct integration of the exponent in \eqref{eq:ck-F-def} gives
\[
   F_\nu(t)
   =
   \left(\frac{C}{\kappa}\right)^{b-1}
   z^{b-1}(1-z)^{-K}(1-rz)^{-a},
\]
and
\[
   L_\nu(t)
   =
   (b-1)\kappa
   \left(\frac{C}{\kappa}\right)^{b-1}
   z^{b-2}(1-z)^{-K}(1-rz)^{1-a}.
\]

To justify the Euler-integral calculation, first take a parameter domain
with \(A<0<C\), choose \(\kappa>0\), and take \(s\) with sufficiently
large real part.  Then the ordinary Laplace integrals converge and establish
the identity analytically.  Both sides are meromorphic functions of the
generic parameters, so the identity extends meromorphically wherever both
sides are defined.  Independently of this global continuation, their
asymptotic expansions at \(s=\infty\) give the stated coefficientwise
formal Laurent-series identity.  Since
\[
   dt=\frac{dz}{\kappa(1-z)},
   \qquad
   e^{-st}=(1-z)^{s/\kappa},
\]
Euler's integral gives the required transforms.  With
\[
   c(s)=\frac{s-\gamma_0}{\kappa}+(a+b-1)r,
\]
one obtains
\begin{align*}
\mathcal L_{\!B}[F_\nu](s)
&=
\frac1\kappa
\left(\frac C\kappa\right)^{b-1}
B\bigl(b,c(s)-b+1\bigr)
{}_2F_1\!\left(a,b;c(s)+1;r\right),\\
\mathcal L_{\!B}[L_\nu](s)
&=
(b-1)
\left(\frac C\kappa\right)^{b-1}
B\bigl(b-1,c(s)-b+1\bigr)
{}_2F_1\!\left(a-1,b-1;c(s);r\right).
\end{align*}
Finally,
\[
   \frac{B(b,d)}{(b-1)B(b-1,d)}=\frac1{b+d-1},
\]
which reduces \eqref{eq:app-transform-quotient} to
\eqref{eq:ck-hypergeometric-return-series}.  Replacing \(\kappa\) by
\(-\kappa\) changes the intermediate parameters but not the Laurent series,
because both branches equal the uniquely determined series
\(\sum_{n\ge0}W_ns^{-n-1}\).
\end{proof}

The appearance of a quotient of contiguous Gauss functions is consistent
with the classical theory of associated Jacobi and Meixner--Pollaczek
recurrences; see, for example, \cite{Masson1988} and the references therein.
The point of the present derivation is different: the quotient is obtained
directly from the return equation \(F_\nu=L_\nu*W\), and the same boundary
identity also reconstructs every terminal height.

\begin{remark}[Computing the return sequence]
The direct recurrence \eqref{eq:motzkin-rec} fills a two-dimensional triangle
and uses \(\Theta(N^2)\) arithmetic operations to obtain the return weights
through length \(N\).  Once the numerator and denominator in
\eqref{eq:ck-series-quotient} have been computed through order \(N\), fast
power-series division recovers \(W_0,\ldots,W_N\) in \(O(M(N))\)
arithmetic operations, where \(M(N)\) is the cost of multiplying series of
length \(N\) \cite{GathenGerhard2013}.  This statement concerns arithmetic
complexity of the final division; it does not include the separate cost of
constructing the two auxiliary series, which are obtained from the quadratic
flow \(X'=Q(X)\) whose standard linearisation is recorded in
Appendix~\ref{app:riccati-collapse}.
\end{remark}


\section{Details for the integer shifted floor}
\label{app:integer-compression}

Throughout this appendix,
\[
   C>0,
   \qquad
   m\in\mathbb Z_{>0},
   \qquad
   \beta_0=(m+1)C,
   \qquad
   \delta=mC.
\]
We write
\[
   \mathcal P_m f
   =Q(x)f'(x)
    +\bigl((\alpha_0-mA)x+\gamma_0-mB\bigr)f(x).
\]

Put \(v(x,t)=x^mw(x,t)\).  Since
\[
   v_t=x^mw_t,
   \qquad
   x^mw_x=v_x-\frac{mv}{x},
\]
multiplication of \eqref{eq:motzkin-pde} by \(x^m\) gives
\begin{align*}
   v_t
   &={}Q(x)\left(v_x-\frac{mv}{x}\right)
      +(\alpha_0x+\gamma_0)v
      +mC\left(\frac{v}{x}-x^{m-1}W(t)\right).
\end{align*}
Using \(Q(x)=Ax^2+Bx+C\), the terms containing \(mCv/x\) cancel, and
we obtain
\[
   v_t=\mathcal P_mv-mCx^{m-1}W(t),
   \qquad
   v(x,0)=x^m.
\]

The free auxiliary functions \(a\) and \(b\) satisfy
\[
   a_t=\mathcal P_ma,
   \qquad a(x,0)=x^m,
\]
\[
   b_t=\mathcal P_mb,
   \qquad b(x,0)=x^{m-1}.
\]
The main text proves the first-entry identity combinatorially.  Its
variation-of-constants form
\[
   v(x,t)
   =a(x,t)-mC\int_0^t b(x,s)W(t-s)\,ds
\]
may also be checked directly by differentiating the right-hand side.  This
is the analytic Duhamel counterpart of Proposition~\ref{prop:integer-first-entry}.

The row-by-row truncation formula follows from the same equation.  Write
\[
   v(x,t)=\sum_{n\ge0}v_n(x)\frac{t^n}{n!}.
\]
Then
\[
   v_{n+1}=\mathcal P_mv_n-mCW_nx^{m-1}.
\]
Since \(v_n=x^mP_n\) has no terms below degree \(m\), the only term of
\(\mathcal P_mv_n\) below degree \(m\) is
\(mCW_nx^{m-1}\).  Therefore
\[
   v_{n+1}=[\mathcal P_mv_n]_{\ge m},
\]
which is \eqref{eq:integer-row-by-row-rule}.

We next evaluate the boundary values of the auxiliary local solutions.  Let
\(X\) be the function defined in Section~\ref{sec:continuous-kernel} by
\[
   X'(t)=Q(X(t)),
   \qquad
   X(0)=0.
\]
More generally, let \(u_r(x,t)\) solve
\[
   \partial_tu_r=\mathcal P_mu_r,
   \qquad
   u_r(x,0)=x^r.
\]
Fix \(t\), put \(y(s)=X(t-s)\), and evaluate \(u_r(y(s),s)\).  Since
\(y'=-Q(y)\),
\[
   \frac{d}{ds}u_r(y(s),s)
   =\bigl((\alpha_0-mA)y(s)+\gamma_0-mB\bigr)u_r(y(s),s).
\]
At \(s=0\), the value is \(X(t)^r\); at \(s=t\), it is \(u_r(0,t)\).
Hence
\[
   u_r(0,t)
   =X(t)^r
    \exp\left(
       \int_0^t
       \bigl[(\alpha_0-mA)X(u)+\gamma_0-mB\bigr]du
    \right).
\]
Taking \(r=m\) and \(r=m-1\), so that \(u_m=a\) and \(u_{m-1}=b\),
gives
\[
   F_m(t)=a(0,t),
   \qquad
   L_m(t)=mC\,b(0,t).
\]

Finally, we justify the second-kind equation.  The shortest auxiliary path
from height \(m-1\) to \(0\) consists of \(m-1\) down-steps and has weight
\(C^{m-1}(m-1)!\).  Therefore
\[
   b(0,t)=C^{m-1}t^{m-1}+O(t^m).
\]
Since \(m\) is an integer, \(L_m(t)=mC\,b(0,t)\) is analytic at the
origin.  Hence
\[
   L_m^{(q)}(0)=0\quad(0\le q\le m-2),
   \qquad
   L_m^{(m-1)}(0)=m!C^m.
\]
For analytic \(K\) and \(W\), repeated differentiation of the convolution
\[
   (K*W)(t):=\int_0^tK(t-s)W(s)\,ds
\]
gives
\[
   \frac{d^m}{dt^m}(K*W)(t)
   =\sum_{q=0}^{m-1}K^{(q)}(0)W^{(m-1-q)}(t)
    +\int_0^tK^{(m)}(t-s)W(s)\,ds.
\]
Applying this formula with \(K=L_m\), all endpoint terms vanish except the
one with \(q=m-1\).  Differentiating
\eqref{eq:integer-first-entry-convolution} exactly \(m\) times therefore
gives
\[
   \partial_t^ma(0,t)
   =m!C^mW(t)
    +mC\int_0^t\partial_t^mb(0,t-s)W(s)\,ds,
\]
and hence \eqref{eq:integer-second-kind-volterra}.

For the coefficient form, write
\[
   F_m(t)=\sum_{n\ge0}f_n\frac{t^n}{n!},
   \qquad
   L_m(t)=\sum_{n\ge0}\lambda_n\frac{t^n}{n!}.
\]
Here \(\lambda_{m-1}=m!C^m\).  Extracting the coefficient of total length
\(N+m\) from \eqref{eq:integer-first-entry-convolution} gives
\[
   f_{N+m}
   =m!C^mW_N
    +\sum_{q=1}^{N}\lambda_{m-1+q}W_{N-q},
\]
which is equivalent to
\eqref{eq:integer-return-coefficient-recursion}.

The generic Gauss quotient of Appendix~\ref{app:continuous-kernel-details}
also simplifies at integer index.  In the remainder of this appendix the
auxiliary series \(a(x,t)\), \(b(x,t)\) no longer occur, and the symbol
\(a\) resumes the meaning it has in
Appendix~\ref{app:continuous-kernel-details}.  Assume now that
\(AC\ne0\) and \(B^2-4AC\ne0\), choose
\(\kappa^2=B^2-4AC\), and put
\[
   a:=\frac{\alpha_0}{A},
   \qquad
   r:=\frac{B+\kappa}{2\kappa},
   \qquad
   c_m(s):=\frac{s-\gamma_0}{\kappa}+(a+m)r.
\]

\begin{corollary}[Integer-index Gauss and incomplete-beta reduction]
\label{cor:integer-incomplete-beta-reduction}
Write \(\mathcal R_m(s)\) for the reciprocal return series
\(\mathcal R(s)\) at boundary index \(\nu=m\).  Then
\begin{equation}
\label{eq:integer-hypergeometric-return-series}
   \mathcal R_m(s)
   =\frac{1}{\kappa c_m(s)}
    \frac{
       {}_2F_1\!\left(a,m+1;c_m(s)+1;r\right)
    }{
       {}_2F_1\!\left(a-1,m;c_m(s);r\right)
    }.
\end{equation}
If \(a\ne1\), this is equivalently
\begin{equation}
\label{eq:integer-logarithmic-derivative}
   \mathcal R_m(s)
   =\frac{1}{\kappa m(a-1)}
    \left.
       \frac{d}{dz}
       \log{}_2F_1\!\left(a-1,m;c_m(s);z\right)
    \right|_{z=r},
\end{equation}
where \(c_m(s)\) is held fixed during differentiation.  For every fixed
\(m\), the Gauss function in the denominator is obtained from one incomplete
beta function by finitely many ordinary derivatives:
\begin{equation}
\label{eq:integer-finite-derivative-reduction}
 {}_2F_1(\lambda,m;c;z)
 =\frac{1}{(m-1)!}
   \frac{d^{m-1}}{dz^{m-1}}
   \left[z^{m-1}{}_2F_1(\lambda,1;c;z)\right],
\end{equation}
with
\begin{equation}
\label{eq:integer-incomplete-beta-base}
 {}_2F_1(\lambda,1;c;z)
 =(c-1)z^{1-c}(1-z)^{c-\lambda-1}
   \mathrm B_z(c-1,\lambda-c+1).
\end{equation}
The identities are understood first for generic parameters and then by
meromorphic continuation, or coefficientwise as formal Laurent-series
identities at \(s=\infty\).
\end{corollary}

\begin{proof}
Formula \eqref{eq:integer-hypergeometric-return-series} is
\eqref{eq:ck-hypergeometric-return-series} with
\(\beta_0/C=m+1\).  The derivative identity
\[
   \frac{d}{dz}{}_2F_1(a-1,m;c;z)
   =\frac{m(a-1)}{c}{}_2F_1(a,m+1;c+1;z)
\]
gives \eqref{eq:integer-logarithmic-derivative}.  Comparing Taylor
coefficients proves \eqref{eq:integer-finite-derivative-reduction}, and
Euler's transformation together with the standard hypergeometric
representation of the incomplete beta function gives
\eqref{eq:integer-incomplete-beta-base}.
\end{proof}


\section{Explicit formulas for the factorised height columns}
\label{app:riccati-collapse}

This appendix supplies the elementary calculations behind
Corollary~\ref{cor:collapse-explicit-form}.  The factorisation itself is
proved directly from the path recurrence in
Theorem~\ref{thm:terminal-height-factorisation}; here we solve the two
one-variable equations and record the elementary degeneration \(C=0\).

Assume first that \(C\ne0\).  To solve
\[
   Y'=A+BY+CY^2,
   \qquad
   Y(0)=0,
\]
write
\begin{equation}
\label{eq:app-collapse-Y-linearisation}
   Y(t)=-\frac1C\frac{u'(t)}{u(t)}.
\end{equation}
Then \(u\) satisfies
\begin{equation}
\label{eq:app-u-ode}
   u''=Bu'-ACu,
   \qquad
   u(0)=1,
   \qquad
   u'(0)=0.
\end{equation}
Conversely, any solution of \eqref{eq:app-u-ode} which is nonzero near the
origin gives a solution of the Riccati equation through
\eqref{eq:app-collapse-Y-linearisation}.

Let
\[
   \kappa^2=B^2-4AC
\]
and let \(D(t)\) be defined by \eqref{eq:collapse-D-def}.  The solution of
\eqref{eq:app-u-ode} is
\begin{equation}
\label{eq:app-collapse-u-D}
   u(t)=e^{Bt/2}D(t).
\end{equation}
Indeed, this expression satisfies the differential equation and the two
initial conditions.  Moreover,
\[
   D'(t)+\frac B2D(t)
   =
   -\frac{2AC}{\kappa}\sinh\frac{\kappa t}{2}.
\]
Substitution in \eqref{eq:app-collapse-Y-linearisation} gives
\[
   Y(t)
   =
   \frac{2A}{\kappa}
   \frac{\sinh(\kappa t/2)}{D(t)},
\]
which is \eqref{eq:collapse-explicit-Y}.  The expression is unchanged when
\(\kappa\) is replaced by \(-\kappa\).

Put \(p=\beta_0/C\).  Since the return series \(W\) satisfies
\eqref{eq:collapse-W-equation},
\eqref{eq:app-collapse-Y-linearisation} gives
\[
   W(t)=e^{\gamma_0t}u(t)^{-p}.
\]
Using \eqref{eq:app-collapse-u-D} yields
\[
   W(t)=e^{(\gamma_0-pB/2)t}D(t)^{-p},
\]
which is \eqref{eq:collapse-explicit-W}.

When \(\kappa=0\), the Taylor limits are
\[
   \frac{\sinh(\kappa t/2)}{\kappa}\longrightarrow\frac t2,
   \qquad
   D(t)\longrightarrow1-\frac B2t,
\]
and \eqref{eq:collapse-double-root} follows.  If \(B^2<4AC\), it is often
more transparent to put \(\tau=\sqrt{4AC-B^2}\).  Then
\[
   D(t)
   =
   \cos\frac{\tau t}{2}
   -\frac B\tau\sin\frac{\tau t}{2},
   \qquad
   Y(t)
   =
   \frac{2A}{\tau}
   \frac{\sin(\tau t/2)}{D(t)}.
\]
The secant and level-step examples in Section~\ref{sec:collapse} are the
specialisations \((A,B,C)=(1,0,1)\) and \((1,1,1)\), respectively.

If \(C=0\), the quadratic term disappears and no linearisation is needed.
For \(B\ne0\),
\[
   Y(t)=\frac AB(e^{Bt}-1),
   \qquad
   W(t)
   =
   \exp\left(
      \gamma_0t
      +\frac{A\beta_0}{B^2}(e^{Bt}-1-Bt)
   \right).
\]
For \(B=0\),
\[
   Y(t)=At,
   \qquad
   W(t)=e^{\gamma_0t+A\beta_0t^2/2}.
\]
In both cases, \(H_k=WY^k\) and
\(w(x,t)=W(t)/(1-xY(t))\).

The boundary method of Section~\ref{sec:continuous-kernel} uses
\[
   X'=AX^2+BX+C,
   \qquad
   X(0)=0.
\]
When \(A\ne0\), the same function \(u\) from
\eqref{eq:app-u-ode} gives the standard linearisation
\[
   X(t)=-\frac1A\frac{u'(t)}{u(t)}.
\]
Comparing this with \eqref{eq:app-collapse-Y-linearisation} gives, for
\(C\ne0\),
\[
   Y(t)=\frac ACX(t).
\]
The same identity follows directly by differentiating \((A/C)X\), so it
also covers \(A=0\).  Hence the return series in the factorised family may
be written as
\[
   W(t)
   =
   \exp\left(
      \gamma_0t
      +\frac{A\beta_0}{C}\int_0^tX(r)\,dr
   \right).
\]
Thus, when \(C\ne0\), the height factor \(Y\) is a rescaled version of the
boundary flow used in Section~\ref{sec:continuous-kernel}.  In the path
formulation of Section~\ref{sec:collapse}, this same relation appears as the
factorisation \(H_k=WY^k\) of all terminal-height columns.

The factorised family also admits a compact kernel for arbitrary initial and
terminal heights.  Let \(w_n^{i\to k}\) denote the total weight of paths of
length \(n\) from height \(i\) to height \(k\), for
\[
   \alpha_h=A(h+1),
   \qquad
   \beta_h=Ch+\beta_0,
   \qquad
   \gamma_h=Bh+\gamma_0.
\]

\begin{proposition}[Full transition kernel in the factorised family]
Let \(\Phi_t(u)\) be the Riccati flow
\[
   \partial_t\Phi_t(u)
   =A+B\Phi_t(u)+C\Phi_t(u)^2,
   \qquad
   \Phi_0(u)=u,
\]
and put
\[
   M(u,t)
   :=\exp\left(
       \gamma_0t+\beta_0\int_0^t\Phi_s(u)\,ds
     \right).
\]
Then the three-variable exponential generating function is
\begin{equation}
\label{eq:full-transition-kernel}
   \sum_{i,k,n\ge0}
      w_n^{i\to k}u^iv^k\frac{t^n}{n!}
   =\frac{M(u,t)}{1-v\Phi_t(u)}.
\end{equation}
At \(u=0\), one has \(\Phi_t(0)=Y(t)\) and \(M(0,t)=W(t)\), so
\eqref{eq:full-transition-kernel} reduces to
\eqref{eq:collapse-w}.
\end{proposition}

\begin{proof}
Let \(\mathcal K(u,v;t)\) denote the left-hand side.  A first step from the
initial height gives the local equation
\[
   \mathcal K_t
   =(A+Bu+Cu^2)\mathcal K_u
    +(\gamma_0+\beta_0u)\mathcal K,
   \qquad
   \mathcal K(u,v;0)=\frac1{1-uv}.
\]
Solving this equation along the Riccati flow gives
\eqref{eq:full-transition-kernel}.
\end{proof}


\section{Fixed-height Jacobi recurrences and finite shifted tails}
\label{app:continued-fraction-details}

This appendix records three optional consequences of the Jacobi viewpoint in
Section~\ref{sec:external-readings}: a finite polynomial reconstruction at
each fixed terminal height, a telescoping hypergeometric formula for those
heights, and the finite continued-fraction correction created by an integer
shifted floor.

Put the formal reciprocal Laurent series
\[
   \mathcal R(s):=\sum_{n\ge0}\frac{W_n}{s^{n+1}},
   \qquad
   \mathcal H_k(s):=\sum_{n\ge0}\frac{w_{n,k}}{s^{n+1}}.
\]
Multiplying the path recurrence \eqref{eq:motzkin-rec} by \(s^{-n-2}\)
and summing over \(n\ge0\) gives
\begin{equation}
\label{eq:app-height-transform-recurrence}
   \beta_k\mathcal H_{k+1}(s)
   =(s-\gamma_k)\mathcal H_k(s)
    -\alpha_{k-1}\mathcal H_{k-1}(s)
    -\mathbf 1_{\{k=0\}},
\end{equation}
where the term involving \(\mathcal H_{-1}\) is absent.

Define polynomial sequences \(\Pi_k\) and \(\Theta_k\) by
\[
   \Pi_0=1,
   \qquad
   \Pi_1=s-\gamma_0,
   \qquad
   \Theta_0=0,
   \qquad
   \Theta_1=1,
\]
and, for \(k\ge1\), by
\begin{equation}
\label{eq:terminal-height-polynomial-recurrence}
   F_{k+1}(s)
   =(s-\gamma_k)F_k(s)-\Lambda_kF_{k-1}(s),
   \qquad
   \Lambda_k=\alpha_{k-1}\beta_{k-1},
\end{equation}
for both \(F=\Pi\) and \(F=\Theta\).

\begin{proposition}[Fixed-height reciprocal series]
If \(\beta_0\beta_1\cdots\beta_{k-1}\ne0\), then
\begin{equation}
\label{eq:terminal-height-series-formula}
   \mathcal H_k(s)
   =
   \frac{
      \Pi_k(s)\mathcal R(s)-\Theta_k(s)
   }{
      \beta_0\beta_1\cdots\beta_{k-1}
   }.
\end{equation}
\end{proposition}

\begin{proof}
At height zero, \eqref{eq:app-height-transform-recurrence} gives
\[
   \beta_0\mathcal H_1
   =(s-\gamma_0)\mathcal R-1
   =\Pi_1\mathcal R-\Theta_1.
\]
Assume that \eqref{eq:terminal-height-series-formula} holds at heights
\(k\) and \(k-1\).  Substitution into
\eqref{eq:app-height-transform-recurrence}, followed by
\(\alpha_{k-1}\beta_{k-1}=\Lambda_k\), gives
\[
   \mathcal H_{k+1}
   =
   \frac{
      \bigl((s-\gamma_k)\Pi_k-\Lambda_k\Pi_{k-1}\bigr)\mathcal R
      -\bigl((s-\gamma_k)\Theta_k-\Lambda_k\Theta_{k-1}\bigr)
   }{
      \beta_0\cdots\beta_k
   }.
\]
The recurrences \eqref{eq:terminal-height-polynomial-recurrence} complete
the induction.
\end{proof}

For generic affine parameters, the preceding polynomial formula can be
combined with the Gauss quotient for the return series to give one closed
hypergeometric expression at every terminal height.  Assume
\(AC\ne0\), \(B^2-4AC\ne0\), choose
\(\kappa^2=B^2-4AC\), and put
\[
   a=\frac{\alpha_0}{A},
   \qquad
   b=\frac{\beta_0}{C},
   \qquad
   r=\frac{B+\kappa}{2\kappa},
   \qquad
   c(s)=\frac{s-\gamma_0}{\kappa}+(a+b-1)r.
\]

\begin{corollary}[Exact reciprocal series at every terminal height]
\label{cor:all-height-hypergeometric}
For every \(k\ge0\),
\begin{equation}
\label{eq:all-height-hypergeometric}
   \mathcal H_k(s)
   =\frac{A^k(a)_k}
          {\kappa^{k+1}(c(s))_{k+1}}
     \frac{
       {}_2F_1\!\left(a+k,b+k;c(s)+k+1;r\right)
     }{
       {}_2F_1\!\left(a-1,b-1;c(s);r\right)
     }.
\end{equation}
This is a formal Laurent-series identity at \(s=\infty\), with the same
choice of branch as in \eqref{eq:ck-hypergeometric-return-series}.
\end{corollary}

\begin{proof}
Let \(\mathcal R^{[j]}(s)\) be the reciprocal return series for paths based
at height \(j\) and constrained never to go below \(j\).  Shifting the
height origin to \(j\) replaces
\[
   a\mapsto a+j,
   \qquad b\mapsto b+j,
   \qquad c(s)\mapsto c(s)+j;
\]
the last identity follows from \(2r-B/\kappa=1\).  Hence
\[
   \mathcal R^{[j]}(s)
   =\frac{1}{\kappa(c(s)+j)}
     \frac{
       {}_2F_1(a+j,b+j;c(s)+j+1;r)
     }{
       {}_2F_1(a+j-1,b+j-1;c(s)+j;r)
     }.
\]
Decomposing a path ending at height \(k\) at its successive first passages
through the levels \(1,\ldots,k\) gives
\[
   \mathcal H_k(s)
   =\left(\prod_{j=0}^{k-1}\alpha_j\right)
      \prod_{j=0}^{k}\mathcal R^{[j]}(s).
\]
The hypergeometric factors telescope, while
\(\prod_{j=0}^{k-1}\alpha_j=A^k(a)_k\) and
\(\prod_{j=0}^{k}(c(s)+j)=(c(s))_{k+1}\), yielding
\eqref{eq:all-height-hypergeometric}.
\end{proof}

We next record the continued-fraction form of the finite virtual strip.  Let
\(\nu=m\in\mathbb Z_{>0}\).  The shifted auxiliary model of
Section~\ref{sec:integer-shift} has Jacobi data
\[
   \widetilde{\mathfrak b}_j
   =Bj+\gamma_0-mB,
   \qquad
   \widetilde\Lambda_{j+1}
   =C(j+1)(Aj+\alpha_0-mA).
\]
Let \(\mathcal T_j(s)\) be the formal reciprocal return series for paths
which start and end at shifted height \(j\) and never go below \(j\).
First-return decomposition gives
\[
   \mathcal T_j
   =\frac{1}{s-\widetilde{\mathfrak b}_j
      -\widetilde\Lambda_{j+1}\mathcal T_{j+1}}.
\]
At heights \(m,m+1,\ldots\), the shifted weights agree with the original
weights after translating the height down by \(m\), and hence
\[
   \mathcal R(s)=\mathcal T_m(s).
\]

For \(0\le j<m\), introduce
\[
   M_j(s)
   =
   \begin{pmatrix}
      0&1\\
      -\widetilde\Lambda_{j+1}&s-\widetilde{\mathfrak b}_j
   \end{pmatrix}.
\]
The fractional-linear map represented by \(M_j\) sends
\(\mathcal T_{j+1}\) to \(\mathcal T_j\).  Write
\[
   M_0(s)M_1(s)\cdots M_{m-1}(s)
   =
   \begin{pmatrix}
      p_m(s)&q_m(s)\\
      r_m(s)&t_m(s)
   \end{pmatrix}.
\]
Then
\[
   \mathcal T_0
   =\frac{p_m\mathcal T_m+q_m}
          {r_m\mathcal T_m+t_m}.
\]
If the couplings
\(\widetilde\Lambda_1,\ldots,\widetilde\Lambda_m\) are nonzero, solving
for \(\mathcal R=\mathcal T_m\) gives
\[
   \mathcal R(s)
   =\frac{q_m(s)-t_m(s)\mathcal T_0(s)}
          {r_m(s)\mathcal T_0(s)-p_m(s)}.
\]
If one of the couplings vanishes, the recurrence splits at that level.
In particular,
\[
   \widetilde\Lambda_m
   =mC(\alpha_0-A),
\]
so the closed-bridge case is exactly the split at the top of the virtual
strip.


\section{Kummer transform for the boundary-perturbed queue}
\label{app:kummer-bd}

This appendix derives the transform \eqref{eq:kummer-transform}.  The
interior coefficients are affine, while the reflecting floor creates a
single defect in the level weight, as explained in
Section~\ref{sec:external-readings}.

For the probability generating function equation \eqref{eq:bd-pgf}, write
\[
   \mathsf q(x)=\theta(1-x),
   \qquad
   \mathsf r(x)=\lambda(x-1),
   \qquad
   \mathsf s(x)=\mu(1-x).
\]
Then
\[
   \mathsf q(0)=\theta>0,
   \qquad
   \mathsf s(0)=\mu>0,
   \qquad
   \nu=\frac{\mu}{\theta},
\]
and the boundary flow is
\[
   X(t)=1-e^{-\theta t}.
\]

Assume \(\mathcal G(x,0)=1\).  The homogeneous factor along a
boundary-hitting characteristic is
\[
   F(t)=X(t)^\nu
        \exp\!\left(\frac{\lambda}{\theta}
        (e^{-\theta t}-1)\right),
\]
because
\[
   \frac{F'(t)}{F(t)}
   =\lambda(X(t)-1)+\frac{\mu(1-X(t))}{X(t)}
   =\mathsf r(X(t))+\frac{\mathsf s(X(t))}{X(t)}.
\]
The same boundary-cancellation calculation as in
Section~\ref{sec:continuous-kernel} gives
\begin{equation}
\label{eq:bd-boundary-volterra}
   F(T)=\int_0^T L(v)p_0(T-v)\,dv,
\end{equation}
where \(p_0(t)=\mathcal G(0,t)\) and
\[
   L(v)=\frac{\mathsf s(X(v))}{X(v)}F(v).
\]
Since \(X(v)=1-e^{-\theta v}\),
\[
   L(v)
   =\mu e^{-\theta v}(1-e^{-\theta v})^{\nu-1}
      \exp\!\left(\frac{\lambda}{\theta}
      (e^{-\theta v}-1)\right).
\]

Using the Laplace transform convention
\[
   \widehat f(s)=\int_0^\infty e^{-st}f(t)\,dt,
\]
we obtain from \eqref{eq:bd-boundary-volterra}
\[
   \widehat p_0(s)=\frac{\widehat F(s)}{\widehat L(s)},
   \qquad \Re s>0.
\]
Put
\[
   a=\frac{\lambda}{\theta},
   \qquad
   r=\frac{s}{\theta},
\]
and set \(y=e^{-\theta t}\).  Then
\[
   \widehat F(s)
   =\frac{e^{-a}}{\theta}
     B(r,\nu+1)
     {}_1F_1(r;r+\nu+1;a),
\]
while
\[
   \widehat L(s)
   =\nu e^{-a}
     B(r+1,\nu)
     {}_1F_1(r+1;r+\nu+1;a).
\]
The quotient of the beta factors is \(1/s\).  Kummer's transformation
\[
   {}_1F_1(\alpha;\gamma;z)
   =e^z{}_1F_1(\gamma-\alpha;\gamma;-z)
\]
therefore yields
\[
   \widehat p_0(s)
   =
   \frac{1}{s}
   \frac{
      {}_1F_1\!\bigl(\nu+1;\nu+1+s/\theta;-a\bigr)
   }{
      {}_1F_1\!\bigl(\nu;\nu+1+s/\theta;-a\bigr)
   },
\]
which is \eqref{eq:kummer-transform}.


\end{document}